\documentclass{amsart}

\usepackage[T1]{fontenc}
\usepackage{beton} \AtBeginDocument{\DeclareFontShape\encodingdefault{ccr}{b}{n}{<->ssub*cmss/sbc/n}{}}
\usepackage[euler-digits]{eulervm} 

\usepackage[margin=1.0in,footskip=0.25in,top=0.75in,left=1.00in,right=1.00in]{geometry}  

\usepackage{amsmath,amsfonts,amsthm,amssymb,mathtools} \usepackage[italicdiff]{physics} \usepackage{graphicx,caption,subcaption,float,tikz,pgfplots,pdfpages,tcolorbox,todonotes,hyperref} \usepackage{pgfplotstable}
\usepackage[inline]{enumitem}
\usepackage{multirow}
\usepackage{booktabs}

\usepackage{algorithm}
\usepackage{algpseudocode}
\algrenewcommand\algorithmicrequire{\textbf{Input:}}
\algrenewcommand\algorithmicensure{\textbf{Output:}}
\algrenewcommand\algorithmiccomment[1]{// #1}

\usetikzlibrary{plotmarks,datavisualization,datavisualization.formats.functions,external}
\makeatletter
\renewcommand{\todo}[2][]{\tikzexternaldisable\@todo[#1]{#2}\tikzexternalenable}
\makeatother
\newcommand{\nomantissatypeset}[1]{\pgfkeys{/pgf/number format/.cd, sci, retain unit mantissa=false}
	\pgfmathprintnumber{#1}
}
\newcommand{\fixedzerofilltypeset}[2]{\pgfkeys{/pgf/number format/.cd, fixed, fixed zerofill, precision=#1}
	\pgfmathprintnumber{#2}
}
\pgfkeys{
	/pgf/data visualization/style sheets/vary mark/.cd,
	1/.style={mark=triangle*},
	2/.style={mark=square*},
	3/.style={mark=pentagon*},
	4/.style={mark=diamond*},
	5/.style={mark=x},
	6/.style={mark=asterisk},
	default style/.style={mark=x},
}

\pgfkeys{
	/pgf/data visualization/style sheets/vary fill/.cd,
	1/.style={fill=black},
	2/.style={fill=red!80!black},
	3/.style={fill=blue!80!black},
	4/.style={fill=green!60!black},
	5/.style={fill=orange!80!black},
	6/.style={fill=black},
	7/.style={fill=black},
	default style/.style={fill=black},
}

\pgfkeys{
	/pgf/data visualization/style sheets/alternate mark/.cd,
	1/.style={mark=square},
	2/.style={mark=x},
	3/.style={mark=square},
	4/.style={mark=x},
	5/.style={mark=square},
	6/.style={mark=x},
	7/.style={mark=square},
	8/.style={mark=x},
	9/.style={mark=square},
	10/.style={mark=x},
	11/.style={mark=square},
	12/.style={mark=x},
	default style/.style={mark=x},
}
\pgfkeys{
	/pgf/data visualization/style sheets/color pair/.cd,
	1/.style={visualizer color=black},
	2/.style={visualizer color=black},
	3/.style={visualizer color=red!80!black},
	4/.style={visualizer color=red!80!black},
	5/.style={visualizer color=blue!80!black},
	6/.style={visualizer color=blue!80!black},
	7/.style={visualizer color=green!60!black},
	8/.style={visualizer color=green!60!black},
	9/.style={visualizer color=orange!80!black},
	10/.style={visualizer color=orange!80!black},
	11/.style={visualizer color=black!60},
	12/.style={visualizer color=black!60},
	default style/.style={black},
}
\pgfkeys{
	/pgf/data visualization/style sheets/alternate dashing/.cd,
	  1/.style=solid,
	  2/.style=solid,
	  3/.style={dash pattern=on 5\pgflinewidth off 2\pgflinewidth},
	  4/.style={dash pattern=on 5\pgflinewidth off 2\pgflinewidth},
	  5/.style={dash pattern=on 1.5\pgflinewidth off 1.5\pgflinewidth},
	  6/.style={dash pattern=on 1.5\pgflinewidth off 1.5\pgflinewidth},
	  7/.style={dash pattern=on 5\pgflinewidth off 1.5\pgflinewidth on 1.5\pgflinewidth off 1.5\pgflinewidth},
	  8/.style={dash pattern=on 5\pgflinewidth off 1.5\pgflinewidth on 1.5\pgflinewidth off 1.5\pgflinewidth},
	  9/.style={dash pattern=on 10\pgflinewidth off 3\pgflinewidth},
	  10/.style={dash pattern=on 10\pgflinewidth off 3\pgflinewidth},
	  11/.style={dash pattern=on 10\pgflinewidth off 2\pgflinewidth on 2\pgflinewidth off 2\pgflinewidth},
	  12/.style={dash pattern=on 10\pgflinewidth off 2\pgflinewidth on 2\pgflinewidth off 2\pgflinewidth},
	  13/.style={dash pattern=on 8\pgflinewidth off 2\pgflinewidth on
		1.5\pgflinewidth off 1.5\pgflinewidth on 1.5\pgflinewidth off
		1.5\pgflinewidth on 1.5\pgflinewidth off 1.5\pgflinewidth}
	  14/.style={dash pattern=on 8\pgflinewidth off 2\pgflinewidth on
		1.5\pgflinewidth off 1.5\pgflinewidth on 1.5\pgflinewidth off
		1.5\pgflinewidth on 1.5\pgflinewidth off 1.5\pgflinewidth}
	default style/.style=solid,
}
\pgfplotsset{compat=1.5}

\theoremstyle{plain}
\newtheorem{thm}{Theorem}
\newtheorem{lem}{Lemma}
\newtheorem{prop}{Proposition}
\newtheorem{cor}{Corollary}

\theoremstyle{definition}
\newtheorem{defn}{Definition}

\theoremstyle{remark}
\newtheorem{rem}{Remark}

\newcommand{\N}{\mathbb{N}}
\newcommand{\Z}{\mathbb{Z}}

\newcommand{\R}{\mathbb{R}}
\newcommand{\C}{\mathbb{C}}

\renewcommand{\P}{\mathbb{P}}

\newcommand{\T}{\mathbb{T}}

\renewcommand{\d}{\mathop{}\!d}

\renewcommand{\vec}{\mathbold}
\newcommand{\restrict}[1]{\rvert_{#1}}

\DeclareMathOperator{\supp}{supp}

\DeclareMathOperator{\Unif}{Unif}

\newcommand{\e}{\mathrm{e}}
\renewcommand{\i}{\mathrm{ i }}

\begin{document}

\title[Sparse spectral methods]{Sparse spectral methods for solving high-dimensional and multiscale elliptic PDEs}
\author[C.\ Gross]{Craig Gross}
\address{Department of Mathematics\\Michigan State University\\619 Red Cedar Road\\East Lansing, MI 48824}
\email{\href{mailto:grosscra@msu.edu}{grosscra@msu.edu}}
\author[M.\ Iwen]{Mark Iwen}
\address{Department of Mathematics\\Michigan State University\\619 Red Cedar Road\\East Lansing, MI 48824}
\address{Department of Computational Mathematics, Science and Engineering\\Michigan State University\\428 S Shaw Lane\\East Lansing, MI 48824}
\email{\href{mailto:iwenmark@msu.edu}{iwenmark@msu.edu}}
\thanks{ This work was supported in part by the National Science Foundation Award Numbers DMS 2106472 and 1912706.}
\date{\today}
\subjclass[2010]{Primary 65N35, 65T40, 35J15;
	Secondary 65D40, 35J05}
\keywords{Spectral methods, sparse Fourier transforms, high-dimensional function approximation, elliptic partial differential equations, compressive sensing, rank-1 lattices}
\begin{abstract}
	In his monograph \emph{Chebyshev and Fourier Spectral Methods}, John Boyd claimed that, regarding Fourier spectral methods for solving differential equations, ``[t]he virtues of the Fast Fourier Transform will continue to improve as the relentless march to larger and larger [bandwidths] continues'' \cite[pg. 194]{boyd_chebyshev_2001}.
This paper attempts to further the virtue of the Fast Fourier Transform (FFT) as not only bandwidth is pushed to its limits, but also the dimension of the problem.
Instead of using the traditional FFT however, we make a key substitution: a high-dimensional, \emph{sparse Fourier transform} (SFT) paired with randomized rank-1 lattice methods.
The resulting \emph{sparse spectral method} rapidly and automatically determines a set of Fourier basis functions whose span is guaranteed to contain an accurate approximation of the solution of a given elliptic PDE.
This much smaller, near-optimal Fourier basis is then used to efficiently solve the given PDE in a runtime which only depends on the PDE's data compressibility and ellipticity properties, while breaking the curse of dimensionality and relieving linear dependence on any multiscale structure in the original problem.
Theoretical performance of the method is established herein with convergence analysis in the Sobolev norm for a general class of non-constant diffusion equations, as well as pointers to technical extensions of the convergence analysis to more general advection-diffusion-reaction equations.
Numerical experiments demonstrate good empirical performance on several multiscale and high-dimensional example problems, further showcasing the promise of the proposed methods in practice.
\end{abstract}
\maketitle

\section{Introduction}\label{sec:introduction}

Consider as a model problem an elliptic PDE with periodic boundary conditions
\begin{equation}
	\label{eq:DiffusionEquation}
	- \nabla \cdot (a \nabla u) = f
\end{equation}
where, for $ \T := \R / \Z$ taken to be the one-dimensional torus, $ a, f: \T^d \rightarrow \R $ are the PDE data, and $ u :\T^d \rightarrow \R $ is the solution.
Herein we propose a two stage method for solving such PDE.  First, we use recently developed SFT methods for high-dimensional functions \cite{gross_sparse_2021} to approximate the Fourier data of both the diffusion coefficient $ a $ and the forcing function $ f $.
So long as the PDE data, $ a $ and $ f $, are well represented by sparse Fourier approximations, we then provide a technique for using the SFT output to find a relatively small number of Fourier coefficients that are guaranteed to reconstruct an accurate approximation of the solution $ u $.  In all, this results in a sublinear-time, curse-of-dimensionality-breaking spectral method for solving non-constant diffusion equations under periodic boundary conditions.
Moreover, the technique presented is theoretically sound, with $ H^1 $ convergence guarantees provided.

These convergence guarantees hinge on a novel analysis of the Fourier-Galerkin representation of a non-constant diffusion operator where we are able to fully characterize the Fourier compressibility of the solution to \eqref{eq:DiffusionEquation} in terms of the Fourier compressibility of the PDE data.
Additionally, we provide algorithmic improvements to the SFT developed in \cite{gross_sparse_2021} that allow the method to run in fully sublinear-time (with respect to the size of the initial frequency set of interest).
This is accompanied by new $ L^\infty $ error guarantees for this SFT which, in addition to the original $ L^2 $ guarantees, allow for the final $ H^1 $ convergence analysis of the spectral method.
We also provide implementations of our methods along with various numerical experiments.
Of special note, we conclude by further extending our methods beyond the simple diffusion equation \eqref{eq:DiffusionEquation} to also apply to multiscale, high-dimensional advection-diffusion-reaction equations including, e.g., the governing equations for flow dynamics in a porous medium used in hydrological modeling \cite{rubio_numerical_2008}.

Solving \eqref{eq:DiffusionEquation} using a traditional Fourier spectral method amounts to replacing the data and the solution with their Fourier series, simplifying the left-hand side into a single Fourier series, matching the Fourier coefficients of both sides, and solving the resulting system of equations for the Fourier coefficients of $ u $.
See Section~\ref{sec:galerkin_spectral_methods} for further explanation of this Galerkin formulation and the related formulations discussed below.

Two main sources of approximation error arise when implementing this technique computationally.
The first is due to truncating the Fourier series involved to a finite number of terms.
The second is due to numerically approximating the Fourier coefficients of the PDE data.
Due to the rich theory of traditional spectral methods, these two sources of error can directly quantify the error of the resulting approximation of $ u $.

\begin{lem}[Strang's lemma, \cite{canuto_spectral_2006}]
	\label{lem:Pseudostrang}
	Let $ u^\mathrm{truncation} $ be the function which has the same Fourier series as $ u $ but truncated in some manner, and $ a^\mathrm{approximate} $ and $ f^\mathrm{approximate} $ be computed using approximations of the Fourier series of $ a $ and $ f $ truncated in the same way as $ u^\mathrm{truncation} $.
	Then the procedure outlined above produces a solution $ u^\mathrm{spectral} $ which satisfies
	\begin{equation*}
		\norm{ u - u^\mathrm{spectral} }_{ H^1 } \lesssim_{a, f} \norm{ u - u^\mathrm{truncation} }_{ H^1 } + \norm{ a - a^\mathrm{approximate} }_{ L^\infty } + \norm{ f - f^\mathrm{approximate} }_{ L^2 }
	\end{equation*}
	where the exact notion of the periodic Sobolev space $ H^1 $ is discussed further in Section~\ref{sec:notation}, and $ \lesssim_{ a, f } $ denotes an upper bound with constants that depend on the PDE data.
\end{lem}

This is a rough simplification of \emph{Strang's lemma} \cite{canuto_spectral_2006}, which is itself a generalization of the well-known \emph{C\'ea's lemma} (the specific version of this lemma used in this paper is presented and proven in Lemma~\ref{lem:StrangsLemma} below).
Effectively, it states that the spectral method solution is optimal up to its Fourier series truncation and the approximation of the PDE data $ a $ and $ f $.
Thus, analyzing convergence reduces to estimating these two errors.

This outline provides the three primary ingredients for this paper: \begin{enumerate}
	\item a truncation method and the resulting error analysis (Section~\ref{sec:stamping_sets}),
	\item a (sparse) Fourier series approximation technique (Sections~\ref{sec:previous_results} and \ref{sec:improvements_with_randomized_lattices}), and
	\item a version of Strang's lemma that ties everything together (Section~\ref{sec:in_the_language_of_canuto_spectral_2006}).
\end{enumerate}
The final method is given in Algorithm~\ref{alg:SparseSFT}.
Its convergence guarantee in Corollary~\ref{cor:SpectralConvergenceWithSFT} shows that the error in approximating $ u $ converges like the (near-optimal) convergence rates of the SFT approximation error of $ a $ and $ f $ in addition to an exponentially decaying term related to the ellipticity properties of $ a $.

The sections preceding the main theoretical analysis listed above include background on sparse spectral methods and motivation for our techniques (Section~\ref{sec:background_and_motivation}), setting the notation and PDE setup (Sections~\ref{sec:notation} and \ref{sec:elliptic_pde_setup} respectively), and the aforementioned Galerkin formulation of our model PDE underpinning the spectral method approach (Section~\ref{sec:galerkin_spectral_methods}).
The paper is closed with a numerics section (Section~\ref{sec:numerics}) describing the implementation of our technique and a variety of numerical experiments demonstrating the theory.

\section{Background and motivation}\label{sec:background_and_motivation}

We now outline some of the previous literature on spectral methods with an emphasis on exploiting sparsity.
Along the way, various shortcomings will arise, and we will use these as opportunities to motivate and explain our approach in the sequel.

\subsection{Convergence and computational complexity}\label{sub:convergence_and_computational_complexity}

Using a $ d $-dimensional FFT (see, e.g., \cite[Section 5.3.5]{plonka_numerical_2018} for details) to compute $ a^\mathrm{approximate} $ and $ f^\mathrm{approximate} $ in the procedure suggested in Lemma~\ref{lem:Pseudostrang} naturally enforces a Fourier series truncation.
A $ d $-dimensional FFT using a tensorized grid of $ K $ uniformly spaced points in each dimension will produce approximate Fourier coefficients indexed by frequencies in the $ d $-dimensional hypercube on the integer lattice $ \Z^d $ of sidelength $ K $ (note that when when we refer to ``bandwidth'' in a multidimensional sense, we are still referring to the sidelength $ K $ of the hypercube containing these integer frequencies).
The cost of each $ d $-dimensional FFT in general requires more than $ K^d $ operations, as does the linear-system solve (in the absence of any sparsity or other tricks).
Thus, not only do traditional Fourier spectral methods suffer from the curse of dimensionality, but even in moderate dimensions, multiscale problems (i.e., PDE data which require very high bandwidth to be fully resolved) can result in intractable computations.

Note that a standard FFT requires more than $ K^d $ operations in the discussion above exactly because we implicitly chose to expand our PDE data and solution with respect to an impractically huge set of $ K^d $ Fourier basis functions there.  What if we instead expand all of $a$, $f$, and $u$ in terms of the union of their individual best possible $s \ll K^d $ Fourier basis functions from this larger set?  Note that doing so would automatically lead to each term on the right hand side of Lemma~\ref{lem:Pseudostrang} becoming related to a nonlinear best $s$-term approximation error with respect to the Fourier basis in the sense of, e.g., Cohen et al \cite{cohen_compressed_2009}.  Furthermore, whenever these errors decayed fast enough in $s$ it would in fact imply that each of $a$, $f$, and $u$ was effectively sparse/compressible in the Fourier basis, allowing the theory of compressive sensing to imply the sufficiency of a small discretization of \eqref{eq:DiffusionEquation}.  Of course, this procedure is not terribly useful in practice unless one can actually rapidly discover the best possible subset of $s \ll K^d $ Fourier basis functions for each function involved above via, e.g., compressive sensing.  

A naive application of standard compressive sensing theory in pursuit of this strategy flounders in at least two ways here, however:  First, though extremely successful at reducing the number of linear measurements needed in order to reconstruct a given function, standard compressive sensing recovery algorithms such as basis pursuit must still individually represent all $K^d$ basis functions (in this simple case) during the function's numerical approximation.  As a result, no dramatic runtime speedups can be expected here without additional modifications.  Second, standard compressive sensing theory also generally requires direct linear measurements (in the form of, e.g., point samples) to be gathered from the function whose sparse approximation one seeks.  In the case of \eqref{eq:DiffusionEquation} this may be trivially possible for both $a$ and $f$, but is not generally possible for the a priori unknown solution $u$ that one aims to compute (at least, not without additional innovations).  Of course these difficulties can be overcome to various degrees even when using standard compressive sensing reconstruction strategies, and at least one such approach for doing so will be discussed below in Section~\ref{sub:links_to_compressed_sensing}.

In this paper, however, we instead circumvent the two difficulties mentioned above by using modified sparse Fourier transform methods.  SFTs \cite{gilbert2002near,iwen2010combinatorial,hassanieh2012simple,gilbert2014recent,bittens2019deterministic,merhi_new_2019} are compressive sensing algorithms which are highly specialized to take advantage of the number theoretic and algebraic structure of the Fourier basis as much as possible.  As a result, SFTs rarely have to consider Fourier basis functions individually during the reconstruction process, and so can simultaneously reduce both their measurement needs \emph{and} computational complexities to effectively depend only on the number of important Fourier series coefficients in the function one aims to approximate.  In the present setting, this means that SFT algorithms will run in sublinear $o(K^d)$-time, more or less automatically sidestepping the reconstruction runtime issues plaguing standard compressive sensing recovery algorithms which must represent each of the $K^d$-basis functions individually as they run.  To circumvent the issues related to not being able to measure the solution $u$ directly, we then use yet another approach.  Instead of attempting to apply compressive sensing methods to $u$ at all, we instead use the more easily discovered most-significant Fourier basis elements of $a$ and $f$ to predict in advance where the most significant Fourier basis elements of $u$ must reside by analyzing the structure of \eqref{eq:DiffusionEquation}.  Of course, once we have discovered which Fourier basis elements are important in representing $u$ in this fashion, standard Galerkin techniques can then be used to solve a small truncated discretization of \eqref{eq:DiffusionEquation} thereafter.  

\subsection{Prior attempts to relieve dependence on bandwidth via SFT-type methods}\label{sub:attempts_to_relieve_dependence_on_bandwidth}

A key work pioneering the use of SFTs in computing solutions to PDEs is due to Daubechies, et al.\ \cite{daubechies_sparse_2007}.
This work mostly focuses on time-dependent, one-dimensional problems where the spectral scheme is formulated as alternating Fourier-projections and time-steps.
Thus, there is no need to impose an a priori Fourier basis truncation on the solution.
The proposed projection step instead utilizes an SFT at each time step to adaptively retain the most significant frequencies throughout the time-stepping procedure.
Time-independent problems like \eqref{eq:DiffusionEquation} can then be handled by stepping in time until a stationary solution is obtained.

A simplified form of this algorithm is shown to succeed numerically in \cite{daubechies_sparse_2007}, and it is also analyzed theoretically in the case where the diffusion coefficient consists of a known, fine-scale mode superimposed over lower frequency terms.
There, the Fourier-projection step can be considered to be fixed.
However, removing the known fine-scale assumption leads to many difficulties, including the possibility of sparsity-induced omissions in early time steps cascading into larger errors later on.
In this paper, on the other hand, we focus on the case of time-independent problems. 
This allows us to utilize SFTs only once initially.  By doing so we avoid the possibility of SFT-induced error accumulation over many time steps. 
The main difficulty in our analysis then becomes determining how the Fourier-sparse representations of the PDE data discovered by high-dimensional SFTs can be used to rapidly find a suitable Fourier representation of the solution.
This takes the form of mixing the Fourier supports of $ a $ and $ f $ into \emph{stamping sets} (discussed in detail in Section~\ref{sec:stamping_sets}) on which we can analyze the projection error of the solution.
In fact, these stamping sets can be viewed as a modification and generalization of the techniques used in the one-dimensional and known fine-scale analysis from \cite{daubechies_sparse_2007}.

\subsection{Attempts to relieve curse of dimensionality}\label{sub:attempts_to_relieve_curse_of_dimensionality}

Many attempts to overcome the curse of dimensionality in Fourier spectral methods for PDE have focused on using basis truncations which allow for an efficient high-dimensional Fourier transform.
One of the most popular techniques is the sparse grid spectral method, which computes Fourier coefficients on the hyperbolic cross \cite{kupka_sparse_1997,bungartz_sparse_2004,gradinaru_fourier_2007,griebel_sparse_2007,shen_sparse_2010,garcke_fast_2014,dung_hyperbolic_2018}.
In general, a sparse grid method reduces the number of sampling points necessary to approximate the PDE data to $ \mathcal{O}(K \log^{ d - 1 }(K)) $, where $ K $ acts as a type of bandwidth parameter.
Algorithms to compute spectral representations using these sparse sampling grids run with similar complexity.
When used in conjunction with spectral methods for solving PDE, these sparse grid Fourier transforms produce solution approximations with error estimates similar to the full $ d $-dimensional FFT-versions reduced by factors only on the order of $ 1 / \log^{ d - 1 }(K) $.

In the context of sparse grid Fourier transforms, these methods compute Fourier coefficients with frequencies on hyperbolic crosses of similar cardinality to the number of sampling points.
These hyperbolic crosses have intimate links with the space of bounded mixed derivative, in the sense that they are the optimal Fourier-approximation spaces for this class.
Thus, sparse grid Fourier spectral methods are particularly apt for problems where the solution is of bounded mixed derivative, as this produces an optimal $ u - u^\mathrm{truncation} $ term in Lemma~\ref{lem:Pseudostrang} above.

Though sparse-grid spectral methods can efficiently solve a variety of high-dimensional problems, there are clear downsides for the types of problems we target in this paper.
While many problems fit the bounded mixed derivative assumption, and therefore have accurate Fourier representations on the hyperbolic cross, the multiscale, Fourier-sparse problems that we are interested are especially problematic.
In fact, since a hyperbolic cross of bandwidth $ K $ contains only those frequencies $ \vec{k} \in \Z^d $ with $ \prod_{ i = 1 }^d \abs{ k_i } = \mathcal{O}(K) $, $ d $-dimensional frequencies active in all dimensions can have only $ \norm{ \vec{ k } }_\infty = \mathcal{O}(K^{ 1/d }) $.
Thus, in a multiscale problem with even one frequency that interacts in all dimensions, a hyperbolic cross is required with a bandwidth exponential in $ d $ to properly resolve the data.
This then forces the traditionally curse-of-dimensionality-mitigating $ \log^{ d - 1 }(K) $ terms characteristic of sparse grid methods to be at least on the order of $ d^{ d - 1 } $.

\subsection{More on high-dimensional Fourier transforms}

As outlined in Section~\ref{sub:attempts_to_relieve_dependence_on_bandwidth} above, this paper uses sparse Fourier transforms to create an adaptive basis truncation suited to the PDE data.
This mimics a similar evolution in the field of high-dimensional Fourier transforms from sparse grids to more flexible techniques \cite{li_trigonometric_2003,dohler_nonequispaced_2010,munthe-kaas_multidimensional_2012,kammerer_interpolation_2012,garcke_fast_2014,kammerer_approximation_2015,plonka_numerical_2018,gross2021deterministic}.
In particular, the high-dimensional sparse Fourier transforms discussed in Section~\ref{sec:previous_results} originate from a link between early high-dimensional quadrature techniques and Fourier approximations on the hyperbolic cross \cite{kammerer_interpolation_2012,kammerer_approximation_2015}.
Instead of sampling functions on sparse grids, these methods sample high-dimensional functions along a rank-1 lattice.
Rank-1 lattices are described by sampling $ M $ points in $ \T^d $ in the direction of a generating vector $ \vec{ z } \in \N^d $, that is, using the sampling set
\begin{equation*}
	\Lambda(\vec{z}, M) := \left\{ \frac{j}{M} \vec{ z } \bmod \vec{ 1 } \mid j \in \{0, \ldots, M - 1\} \right\}.
\end{equation*}

So long as a rank-1 lattice satisfies certain properties with respect to a frequency space of interest $ \mathcal{I} \in \Z^d $, these sampling points are sufficient to compute the Fourier coefficients of a function on $ \mathcal{I} $ with a length-$ M $ univariate FFT.
Though many references take $ \mathcal{I} $ to be the hyperbolic cross to leverage the well-studied regularity properties and cardinality bounds similarly enjoyed in the sparse-grid literature, rank-1 lattice results are available for arbitrary frequency sets.
The computationally efficient extension of these techniques via sparse Fourier transforms in \cite{gross_sparse_2021} as well as the randomization trick presented in Section~\ref{sec:improvements_with_randomized_lattices} take this frequency set flexibility to its limit, allowing $ \mathcal{I} $ to be the a priori unknown set of the most important Fourier coefficients of the function to be approximated.
This again suggests the applicability of these methods over sparse grid (or other non-sparsity exploiting) Fourier transforms in the context of multiscale problems involving even a small number of Fourier coefficients in extremely high dimensions.

\subsection{Additional links to compressive sensing}\label{sub:links_to_compressed_sensing}

As discussed above, the SFT literature overlaps considerably with the language and techniques of compressive sensing.
As detailed in Section~\ref{sec:previous_results} below, the high-dimensional SFT we use in this paper provides error bounds with best $ s $-term approximation, compressive-sensing-type error guarantees \cite{cohen_compressed_2009}.
As a result, the Fourier coefficients of the PDE data are approximated with errors depending on the compressibility of their true Fourier series, and then the compressibility of the PDE's solution in the Fourier basis is inferred from the Fourier compressibility of the data in a direct and constructive fashion.

Another very successful line of work, however, aims to more directly apply standard compressive sensing reconstruction methods to the general spectral method framework for solving PDEs.
Referred to as CORSING \cite{brugiapaglia_compressed_2015, brugiapaglia_compressed_2016, brugiapaglia_theoretical_2018, brugiapaglia_waveletfourier_2020,brugiapaglia_sparse_2021}, these techniques use compressed sensing concepts to recover a sparse representation of the solution to the system of equations derived from the (Petrov-)Galerkin formulation of a PDE.
These methods have been further extended to the case of pseudospectral methods in \cite{brugiapaglia_compressive_2020}, in which a simpler-to-evaluate matrix equation is subsampled and used as measurements for a compressive sensing algorithm (as an aside, \cite{brugiapaglia_compressive_2020} and discussions with the author served as a primary inspiration for this paper).
This compressive spectral collocation method works by finding the largest Fourier-sine coefficients of the solution with frequencies in the integer hypercube with bandwidth $ K $ by applying Orthogonal Matching Pursuit (OMP) on a set of samples of the PDE data.
By using OMP, the method is able to succeed with measurements on the order of $ \mathcal{O}(d \exp(d) s \log^3(s) \log(K)) $ where $ s $ is the imposed sparsity level of the solution's Fourier series.
Thus, while the $ \mathcal{O}(K^d) $ dependence from a traditional Fourier (pseudo)spectral method is avoided and the method adapts well to large bandwidths, the curse of dimensionality is still apparent.

In the preparation of this paper, the authors became aware of an improvement on \cite{brugiapaglia_compressive_2020} that addresses the curse of dimensionality and is therefore well-suited for similar types of problems discussed in this paper.
In \cite{wang_compressive_2022}, the approach of approximating Fourier-sine coefficients on a full hypercube is replaced with approximating Fourier coefficients on a hyperbolic cross.
This has the effect of converting the linear dependence on $ d $ in the sampling complexity to a $ \log(d) $ due to cardinality estimates of the hyperbolic cross.
However, the $ \exp(d) $ term is refined using a different technique.
The key theoretical ingredient for being able to apply compressive sensing to these problems is bounding the Riesz constants of the basis functions that result after applying the differential operator \cite{brugiapaglia_sparse_2021}.
A careful estimation of these constants on the Fourier basis on the hyperbolic cross is able to entirely remove the exponential in $ d $ dependence, leading to a sampling complexity on the order of $ \mathcal{O}(C_a s \log(d) \log^3(s) \log(K)) $, where $ C_a $ involves terms depending on ellipticity and compressibility properties of $ a $.
Notably, this estimation procedure has connections to our stamping set techniques described in Section~\ref{sec:stamping_sets}.

On the other hand, though focusing on the hyperbolic cross in compressive spectral collocation breaks the curse of dimensionality in the sampling complexity, the method still suffers from the inability to generalize to multiscale problems or generic frequency sets of interest like those described in \ref{sub:attempts_to_relieve_curse_of_dimensionality}.
Additionally, as mentioned in Section~\ref{sub:links_to_compressed_sensing}, the compressive-sensing algorithm used for recovery (in this case OMP) suffers from a computational complexity on the order of the cardinality of the truncation set of interest.
For the hyperbolic cross, this is still exponential in $ \log(d) $.
Finally, the error estimates are presented in terms of the compressibility of the Fourier series of the solution $ u $, which may not be known a priori from the PDE data.
We expect that there may be some way to link our stamping theory and convergence estimates with the compressive sensing theory to refine and generalize both approaches.

\section{Notation}\label{sec:notation}

Define the one-dimensional torus to be $ \T := \R / \Z $.
Unless otherwise stated, all functions are complex-valued and defined on the torus $ \T^d $.
For example, we take the inner product for $ u,v \in L^2 := L^2(\T^d; \C) $ to be
\begin{equation*}
	\langle u, v \rangle_{ L^2 } := \int_{ \T^d } u(\vec{ x }) \overline{v}(\vec{ x }) \d \vec{ x }.
\end{equation*}
Additionally, unless otherwise stated, all multiindexed infinite sequences are complex-valued and indexed on $ \Z^d $.
For example, we take the inner product for $ \hat u, \hat v \in \ell^2 := \ell^2(\Z^d; \C) $ to be
\begin{equation*}
	\langle \hat u, \hat v \rangle_{ \ell^2 } := \sum_{ \vec{ k } \in \Z^d } \hat u_\vec{ k } \overline{\hat v}_\vec{ k }.
\end{equation*}
All finite length vectors/tensors will be denoted in boldface and when required, will be implicitly extended to larger index sets by taking on the value zero wherever they are not originally defined.
We also denote the complex-valued finite-length vectors or infinite-length sequences supported on a set $ \mathcal{D} $ as $ \C^\mathcal{D} $.
Since sparse approximations will be an important tool in our final algorithm, we also define the best $ s $-term approximation of a sequence $ \hat{ u } $ as $ \hat{ u } $ restricted to its $ s $ largest magnitude entries and denote this as $ \hat{ u }_s^\mathrm{opt} $.

We now define periodic Sobolev spaces (see also \cite[Section 2.1]{brugiapaglia_waveletfourier_2020} and \cite[Appendix A.2.2]{kupka_sparse_1997}).
\begin{defn}
	\label{def:PeriodicSobolevSpaces}
	For $ u \in L^2 $ and $ \vec{ \alpha } \in \N_0^d $ a multiindex, if there exists a $ v \in L^2 $ such that 
	\begin{equation*}
		\langle v, \phi \rangle_{ L^2 } = (-1)^{ |\vec{ \alpha }| } \langle u, \partial^\vec{ \alpha } \phi \rangle_{ L^2 } \qquad \text{for all $ \phi \in C^\infty \subset L^2$},
	\end{equation*}
	we call $ v $ the \emph{weak $ \vec{ \alpha } $ derivative of $ u $}, and write $ \partial^\vec{ \alpha } u := v $.
	We define the inner product
\begin{equation*}
	\langle u, v \rangle_{ H^1 } := \langle u, v \rangle_{ L^2 } + \int_{ \T^d } \nabla u( \vec{ x }) \cdot \overline{ \nabla v } ( \vec{ x } ) \d \vec{ x },
	\end{equation*}
	(where all derivatives are taken in the weak sense) and have the associated norm $ \norm{ u }_{ H^1 } := \sqrt{ \langle u, u \rangle_{ H^1 } } $.
	The \emph{periodic Sobolev space $ H^1 $} is defined as $ H^1 := \{u \in L^2 \mid \|u\|_{ H^1 } < \infty \} $.
\end{defn}

In order to set our notation for Fourier coefficients and series, we first note the density of trigonometric monomials in $ L^2 $ and $ H^1 $.

\begin{thm}
	\label{thm:DensityOfTrigMonomials}
	The space of all infinitely differentiable periodic functions $ C^\infty $ is dense in $ L^2 $ and $ H^1 $.
	In particular, space of trigonometric monomials $ \{ e_\vec{ k }(\vec{ x }) := \e^{ 2 \pi \i \vec{ k } \cdot \vec{ x }} \in C^\infty \mid k \in \Z^d \} $ is a basis for $ C^\infty $, an orthonormal basis for $ L^2 $, and an orthogonal basis for $ H^1 $.
\end{thm}

\begin{defn}
	For any $ u \in L^1 $, and any $ \vec{ k } \in \Z^d $, we define the \emph{$ \vec{ k } $th Fourier coefficient}
	\begin{equation*}
		\hat{u}_\vec{ k } = \langle u, e_\vec{ k } \rangle_{ L^2 } = \int_{ \T^d } u(\vec{ x }) \e^{ -2 \pi \i \vec{ k } \cdot \vec{ x } } \d \vec{ x }.
	\end{equation*}
	If $ u \in L^2 $, the orthonormality of the trigonometric monomials in Theorem~\ref{thm:DensityOfTrigMonomials} allows us to write the \emph{Fourier series for $ u $},
	\begin{equation*}
		u(\vec{ x }) = \sum_{ \vec{ k } \in \Z^d } \hat u_\vec{ k } e_{ \vec{ k } }(\vec{ x }).
	\end{equation*}
\end{defn}

We also note the well-known Plancherel's identity for use later.

\begin{prop}[Plancherel's identity]
	\label{prop:Plancherel}
	If $ u \in L^2 $, then $ \hat u \in \ell^2 $ with $ \| u \|_{ L^2 } = \| \hat u \|_{ \ell^2 } $.
	If $ v \in L^2 $, then $ \langle u, v \rangle_{ L^2 } = \langle \hat u, \hat v \rangle_{ \ell^2 } $.
\end{prop}

\begin{defn}
	We additionally define the \emph{mean-zero periodic Sobolev space $ H $} as $ H^1 / \R $ where the representative $ u $ is chosen so that $ \hat u_{ \vec{ 0 } } = 0 $, endowed with the inner product\footnote{
		note that by Proposition~\ref{prop:Plancherel}, $ \langle u, v \rangle_{ H } \simeq \langle u, v \rangle_{ H^1 } $ for $ u, v \in H $.
	}
	\begin{equation*}
		\langle u, v \rangle_{ H } := \int_{ \T^d } \nabla u(\vec{ x }) \cdot \overline{ \nabla v }( \vec{ x } ) \d \vec{ x }.
	\end{equation*}
\end{defn}

In the sequel, we will often consider restrictions in frequency space denoted by, e.g., $ \hat u \restrict{ \mathcal{D} } $, where $ \mathcal{D} \subset \Z^d $.
We will simultaneously consider this to be an element of $ \C^{ \mathcal{D} } $ and a complex valued sequence on $ \Z^d $ with zero entries on $ \Z^d \setminus \mathcal{D} $.
When $ \hat u $ represents the Fourier coefficients of a function $ u $, we define the associated restriction
\begin{equation*}
	u\restrict{ \mathcal{D} } := \sum_{ \vec{ k } \in \Z^d } \left( \hat u\restrict{ \mathcal{D} } \right)_\vec{ k } e_\vec{ k } = \sum_{ \vec{ k } \in \mathcal{D} } \hat u_\vec{ k } e_\vec{ k },
\end{equation*}
where the fact that $ \mathcal{D} \subset \Z^d $ is treated as a set of frequencies indicates that we are restricting $ u $ in frequency, not space.
Given a hatted sequence $ \hat{ v } $ or vector $ \hat{ \vec{ v } } $, the associated function with Fourier series $ \sum_{ \vec{ k } \in \Z^d } \hat{ v }_\vec{ k } e_\vec{ k } $ will always be implicitly labeled using the non-hatted, roman font letter (in this example, $ v $).

\section{Elliptic PDE setup}\label{sec:elliptic_pde_setup}

We begin with a model elliptic partial differential equation.
\begin{defn}
	\label{def:EllipticPDE}
	For some $ a: \T^d \rightarrow \R $ sufficiently smooth, define the \emph{linear, elliptic partial differential operator in divergence form} $ \mathcal{L}[a]:C^2 \rightarrow C^0 $ by
	\begin{equation*}
		\mathcal{L}[a]u = - \nabla \cdot \left( a \nabla u \right).
	\end{equation*}
	If for some $ f: \T^d \rightarrow \R $ sufficiently smooth, $ u \in C^2 $ satisfies
	\begin{equation}
		\mathcal{L}[a]u = f, \tag{SF} \label{eq:StrongPDE}
	\end{equation}
	we say that $ u $ \emph{solves the given elliptic PDE with periodic boundary conditions in the strong form}.

	Now, after multiplying by the complex conjugate of a test function $ v \in H^1(\T^d) $ and integrating by parts, we define the bilinear form associated to $ \mathcal{L}[a] $ as $ \mathfrak{L}[a]: H^1 \times H^1 \rightarrow \C $ with
	\begin{equation*}
		\mathfrak{L}[a](u, v) := \int_{ \T^d } a(\vec{ x }) \nabla u(\vec{ x }) \cdot \overline{ \nabla v }(\vec{ x })  \d \vec{ x },
	\end{equation*}
	and we say that $ u \in H^1 $ \emph{solves the given elliptic PDE with periodic boundary conditions in the weak form} if 
	\begin{equation}
		\mathfrak{L}[a](u, v) = \langle f, v \rangle_{ L^2 } \quad \text{for all } v \in H^1. \tag{WF} \label{eq:WeakPDEWholeTestSpace}
	\end{equation}
	For our purposes, we will take $ a \in L^\infty(\T^d; \R) $, and $ f \in L^2(\T^d; \R) $.
\end{defn}

By the conditions specified in the Lax-Milgram theorem (see, e.g., \cite{evans_partial_2010}), we are guaranteed that a unique mean-zero solution to \eqref{eq:WeakPDEWholeTestSpace} exists so long as the right-hand side and test space is also mean-zero.
See \cite[Proposition~2.1]{brugiapaglia_waveletfourier_2020} for a more specific formulation in our setting and its proof.
\begin{prop}
	\label{prop:ExistenceUniquenessStability}
	For $ a \in L^\infty(\T^d; \R) $, $ \mathfrak{L}[a] $ is continuous with continuity constant $ \beta \leq \norm{ a }_{ L^\infty } $, that is
	\begin{equation}
		\label{eq:Continuity}
		\abs{ \mathfrak{L}[a](u, v) } \leq \beta \norm{ u }_{ H } \norm{ v }_{ H } \qquad \text{for all $ u, v \in H $}.
	\end{equation}
	Additionally, if $ a( \vec{ x }) \geq a_{ \mathrm{min} } > 0 $ a.e.\ on $ \T^d $, then $ \mathfrak{L}[a] $ is also coercive with coercivity constant $ \alpha \geq a_\mathrm{min} $, that is
	\begin{equation}
		\label{eq:Coercivity}
		\abs{ \mathfrak{L}[a](u, u) } \geq \alpha \norm{ u }_{ H }^2 \qquad \text{for all $ u \in H $}.
	\end{equation}
	Under conditions \eqref{eq:Continuity} and \eqref{eq:Coercivity}, if $ f \in L^2(\T^d; \R) $ is mean-zero, that is, $ \hat f_{ 0 } = 0 $, then \eqref{eq:WeakPDEWholeTestSpace} has unique, mean-zero solution $ u \in H $ satisfying
	\begin{equation}
		\label{eq:StabilityEstimate}
		\norm{ u }_{ H } \leq \frac{\norm{ f }_{ L^2 }}{\alpha}.
	\end{equation}
\end{prop}

\section{Galerkin spectral methods}\label{sec:galerkin_spectral_methods}

By Theorem~\ref{thm:DensityOfTrigMonomials}, it is equivalent to replace the weak PDE \eqref{eq:WeakPDEWholeTestSpace} by
\begin{equation*}
	\mathfrak{L}[a](u, e_\vec{ k }) = \langle f, e_\vec{ k } \rangle_{ L^2 } =: \hat f_\vec{ k } \quad \text{for all } \vec{ k } \in \Z^d. \end{equation*}
Rewriting the bilinear form on the left-hand side and using the Fourier series representations of $ a $ and $ u $, we obtain
\begin{align*}
	\mathfrak{L}[a](u, e_\vec{ k }) 
		&= \sum_{ \vec{ l }_1, \vec{ l }_2 \in \Z^d } \hat a_{ \vec{ l }_1 } \hat u_{ \vec{ l }_2 } \int_{ \T^d } e_{ \vec{ l_1 } }( \vec{ x } ) \nabla e_{ \vec{ l_2 } }( \vec{ x } ) \cdot \overline{ \nabla e_{ \vec{ k } } }( \vec{ x } ) \d \vec{ x } \\
		&= \sum_{ \vec{ l }_1, \vec{ l }_2 \in \Z^d } (2 \pi)^2 (\vec{ l_2 } \cdot \vec{ k }) \hat a_{ \vec{ l }_1 } \hat u_{ \vec{ l }_2 } \delta_{ \vec{ l }_1, \vec{ k } - \vec{ l }_2 } \\
		&= \sum_{ \vec{ l } \in \Z^d } (2 \pi)^2 (\vec{ l } \cdot \vec{ k }) \hat a_{ \vec{ k } - \vec{ l } } \hat u_{ \vec{ l } } \\
		&=: (L[\hat a] \hat u)_\vec{ k },
\end{align*}
where $ L[\hat a] $ is an operator in $ \ell^2 $.
This leads to the \emph{Galerkin form} of our PDE,
\begin{equation}
	L[\hat a] \hat u = \hat f \tag{GF} \label{eq:GalerkinForm}.
\end{equation}

The computational advantages of \eqref{eq:GalerkinForm} are clear.
By numerically approximating $ \hat a $ and $ \hat f $ (thereby also truncating $ L[\hat a] $), we arrive at a discretized, finite system of equations that can be solved for the Fourier coefficients of our solution.

We will use a fast sparse Fourier transform (SFT) for functions of many dimensions to approximate our PDE data which then leads to a sparse system of equations that we can quickly solve to approximate $ \hat u $.
This SFT will use the values of $ a $ and $ f $ at equispaced nodes on a randomized rank-1 lattice in $ \T^d $, and therefore, our technique is effectively a pseudospectral method where the discretization of the solution space $ \{ \hat u \mid u \in H \} $ is adapted to the PDE data.

Before we move to the detailed discussion of this SFT, we provide a more detailed analysis of the Galerkin operator in Section~\ref{sec:stamping_sets} to help us analyze the resulting spectral method.
But first, we note that $ L[\hat a] $ also captures the behavior of $ \mathfrak{L}[a] $ as a bilinear form.
\begin{prop}
	\label{prop:GalerkinBilinearForm}
	For $ \hat u, \hat v \in \ell^2 $ with $ u, v \in H $, 
	\begin{equation*}
		\mathfrak{L}[a](u, v) = \langle L[\hat a] \hat u, \hat v \rangle_{ \ell^2 }.
	\end{equation*}
\end{prop}
\begin{proof}
	By the Fourier series representation of $ v $,
	\begin{equation*}
		\mathfrak{L}[a](u, v) = \sum_{ \vec{ k } \in \Z^d } \mathfrak{L}[a](u, e_\vec{ k }) \overline{\hat v}_\vec{ k } = \sum_{ \vec{ k } \in \Z^d } \left( L[\hat a] \hat u \right)_{ \vec{ k } } \overline{ \hat v }_{ \vec{ k } } = \langle L[\hat a] \hat u, \hat v \rangle_{ \ell^2 }.
	\end{equation*}
\end{proof}

\section{Stamping sets and truncation analysis}\label{sec:stamping_sets}

Notably, \eqref{eq:GalerkinForm} gives us insight into the frequency support of $ \hat u $.
The structure outlined in the following proposition is crucial in constructing a fast spectral method that exploits Fourier-sparsity.
\begin{prop}
	\label{prop:SolutionSupport}
	For any set $ F \subset \Z^d $ and $ N \in \N_0 $, recursively define the sets
	\begin{equation}
		\label{eq:StampSet}
		\begin{aligned}
			\mathcal{S}^N[\hat a](F) &:= 
			\begin{cases}
				F & \text{if } N = 0 \\
				\mathcal{S}^{ N - 1 }[\hat a](F) + \supp(\hat a)  & \text{if } N > 0
			\end{cases},\\
			\mathcal{S}^\infty[\hat a](F) &:= \bigcup_{ N = 0 }^\infty \mathcal{S}^N[\hat a](F),
		\end{aligned}
	\end{equation}
	where here, we addition is the Minkowski sum of sets.
	Under the conditions of Proposition~\ref{prop:ExistenceUniquenessStability}, $ \supp(\hat u) \subset \mathcal{S}^\infty[\hat a](\supp(\hat f)) $.
\end{prop}
\begin{proof}
	The fact that $ a $ is strictly positive implies that $ \hat a_\vec{ 0 } \neq 0 $, and the fact that $ a $ is real implies $ \supp(\hat a) = - \supp(\hat a) $.
	Now, for any $ \vec{ k } \in \mathcal{\Z}^d \setminus \{\vec{ 0 }\} $, we may rearrange the equality $ (L[\hat a]\hat u)_\vec{ k } = \hat f_\vec{ k } $ to obtain
	\begin{align*}
		\hat u_{ \vec{ k } } 
			&= \frac{\hat f_\vec{ k } - \sum_{ \vec{ l } \in (\{\vec{ k }\} + \supp(\hat a)) \setminus \{ \vec{ k } \}} (2\pi)^2 (\vec{ l } \cdot \vec{ k }) \hat a_{ \vec{ k } - \vec{ l }  } \hat u_{ \vec{ l } }}{ (2\pi)^2 (\vec{ k } \cdot \vec{ k }) \hat a_\vec{ 0 }} \\
			&= \frac{\hat f_\vec{ k } - \sum_{ \vec{ l } \in \supp(\hat a) \setminus \{ \vec{ 0 } \}} (2\pi)^2 (\vec{ k } \cdot \vec{ k } - \vec{ l } \cdot \vec{ k }) \hat a_{\vec{ l }} \hat u_{ \vec{ k } - \vec{ l } }}{ (2\pi)^2 (\vec{ k } \cdot \vec{ k }) \hat a_\vec{ 0 }}.
	\end{align*}

	Thus, $ \hat u_\vec{ k } $ explicitly depends only on the values of $ \hat u $ on $ \mathcal{S}^1[\hat a](\{ \vec{ k } \}) \setminus \{ \vec{ k } \}  $, which themselves then depend only on values of $ \hat u $ on $ \mathcal{S}^2[\hat a]( \{ \vec{ k } \}) $, and so on.
	This decouples the system of equations $ L[\hat a] \hat u $ into a disjoint collection of systems of equations, one for each class of frequencies $ \mathcal{S}^\infty[\hat a]( \{\vec{ k }\} ) $.  
	Since Proposition~\ref{prop:ExistenceUniquenessStability} implies that $ \hat v = 0 $ is the unique solution of $ L[\hat a] \hat v = 0 $, the unique solution of the system of equations for $ \hat u $ on $ \mathcal{S}^\infty[\hat a](\{ \vec{ k } \}) $ for any $ \vec{ k } \notin \supp(\hat f) $ is $ \hat u\rvert_{ \mathcal{S}^\infty[\hat a](\{ \vec{ k } \}) } = 0 $.
	Therefore, $ \supp \hat u \subset \mathcal{S}^\infty[\hat a]( \supp(\hat f)) $ as desired.
\end{proof}

In what follows, when the set $ F $ and Fourier coefficients $ \hat{a} $ are clear from context, we suppress them in the notation given by \eqref{eq:StampSet} so that $ \mathcal{S}^N := \mathcal{S}^N[\hat{a}](F) $.
Intuitively, we can imagine constructing $ \mathcal{S}^N $ by first creating a ``rubber stamp'' in the shape of $ \supp(\hat{ a }) $.
This rubber stamp is then stamped onto every frequency in $ F =: \mathcal{S}^{0} $ to construct $ \mathcal{S}^1 $.
Then, this process is repeated, stamping each element of $ \mathcal{S}^1 $ to produce $ \mathcal{S}^2 $, and so on.
For this reason, we will colloquially refer to these as ``stamping sets.''
Figure~\ref{fig:Stamp} gives an example of this stamping procedure for $ d = 2 $.

\tikzdvdeclarestylesheetcolorseries{stamp hue}{hsb}{0.65,1,1}{0,-.15,0}
\newcommand{\stamp}[1]{
	\begin{tikzpicture}[scale=0.75]
		\datavisualization[	scientific axes,
						all axes={unit length=2cm per 10 units, ticks={none}},
						visualize as scatter/.list={0,...,#1},
						style sheet=vary hue,
						style sheet=vary mark,
]
	data [format = table, read from file=Figures/stamp.csv]
	info {
		\node at (0, 2.3) [above] {$\mathcal{S}^{#1}[\hat{ a }]\left(\supp(\hat{ f })\right)$};
	};
	\end{tikzpicture}
}
\begin{figure}[ht]
	\begin{center}
	\begin{tikzpicture}[scale=0.75]
		\datavisualization[	scientific axes,
						all axes={unit length=2cm per 10 units, ticks={none}},
						visualize as scatter/.list={-1},
						every visualizer/.style={mark=star},
						]
		data [format = table, read from file=Figures/stamp.csv]
		info {
			\node at (0, 2.3) [above] {$\supp(\hat{ a })$};
		};
	\end{tikzpicture}
	\begin{tikzpicture}[scale=0.75]
		\datavisualization[	scientific axes,
						all axes={unit length=2cm per 10 units, ticks={none}},
						visualize as scatter/.list={0},
						every visualizer/.style={mark=square*},
						style sheet=vary hue,
						style sheet=vary mark,
						]
	data [format = table, read from file=Figures/stamp.csv]
	info {
		\node at (0, 2.3) [above] {$\supp(\hat{ f }) = \mathcal{S}^0[\hat{ a }]\left(\supp(\hat{ f })\right)$};
	};
	\end{tikzpicture}
	\stamp{1}
	\\[1ex]
	\stamp{2}
	\begin{tikzpicture}[scale=0.75]
		\datavisualization[	scientific axes,
						all axes={unit length=2cm per 10 units, ticks={none}},
						visualize as scatter/.list={0,...,3},
						style sheet=vary mark,
						style sheet=vary hue,
						legend=right,
						0={label in legend={text={$N=0$}}},
						1={label in legend={text={$N=1$}}},
						2={label in legend={text={$N=2$}}},
						3={label in legend={text={$N=3$}}},
						]
	data [format = table, read from file=Figures/stamp.csv]
	info {
		\node at (0, 2.3) [above] {$\mathcal{S}^{3}[\hat{ a }]\left(\supp(\hat{ f })\right)$};
	};
	\end{tikzpicture}
	\end{center}
	\caption{New frequencies in each stamping level up to $ N = 3 $ where $ N = 0 $ is $ \supp(\hat{ f }) $.}
	\label{fig:Stamp}
\end{figure}
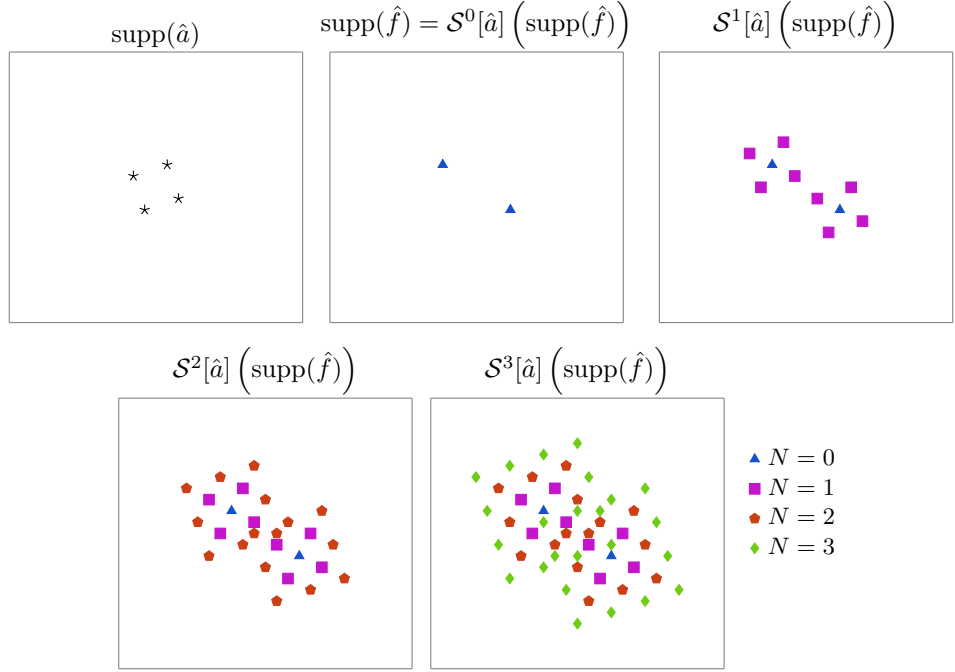
A key approach of our further analysis will be analyzing the decay of $ \hat{ u } $ on successive stamping levels.
The stamping level will become the driving parameter in the spectral method rather than bandwidth in a traditional spectral method.
Before moving onto this analysis however, we provide an upper bound for the cardinality of the stamping sets.
This will ultimately be used to upper bound the computational complexity of our technique.
The proof of this bound is given in Appendix~\ref{sec:stamp_set_cardinality_bound}.

\begin{lem}
	\label{lem:StampSizeUpperBound}
	Suppose that $ \vec{ 0 } \in \supp(\hat{ a }) $, $ \supp(\hat{ a }) = -\supp(\hat{ a }) $, and $ \abs{ \supp(\hat{ f }) } \leq \abs{ \supp(\hat{ a }) } = s $.
	Then
	\begin{equation*}
		\abs{ \mathcal{S}^N[\hat{ a }](\supp (\hat{ f }))} \leq 7 \max(s, 2N + 1)^{ \min(s, 2N + 1)  }.
	\end{equation*}
\end{lem}

Proposition~\ref{prop:SolutionSupport} gives us a natural way to consider truncations of the solution $ u $ in frequency space.
We will use these truncations to discretize the Galerkin formulation \eqref{eq:GalerkinForm} in Section~\ref{sec:in_the_language_of_canuto_spectral_2006} below.
In order to analyze the error in the resulting spectral method algorithm, we will need quantitative bounds on how the solution decays outside of the frequency sets $ \mathcal{S}^N := \mathcal{S}^N[\hat{ a }](\supp (\hat{ f })) $.
For $ \mathcal{S}^N $ to be finite, we assume in this section that $ \supp \hat{ a } $ and $ \supp \hat{ f } $ are finite.
This assumption will be lifted later via Lemma~\ref{lem:RestrictedSupportRecovery}.

We begin with a technical result regarding the interplay between $ L[\hat{ a }] $ and the supports of vectors that it acts on.

\begin{prop}
	\label{prop:LExpansion}
	For any $ \hat{ v } $ with $ \supp(\hat{ v }) \subset \mathcal{S}^n \setminus \mathcal{S}^{ n - 1 } $, $ \supp(L[\hat{ a }]\hat{ v }) \subset \mathcal{S}^{ n + 1 } \setminus \mathcal{S}^{ n - 2 } $.
\end{prop}
\begin{proof}
	For any $ \vec{k} \in \Z^d $, consider
	\begin{align*}
		\left(L[\hat{ a }] \hat{ v }\right)_\vec{k}
			&= \sum_{ \vec{l} \in \Z^d} (2\pi)^2 ( \vec{ l } \cdot \vec{ k }) \hat{ a }_{ \vec{ k } - \vec{ l } } \hat{ v }_\vec{ l }\\
			&= \sum_{ \vec{l} \in (\{ \vec{ k } \} - \supp(\hat{ a })) \cap \supp( \hat{ v })} (2\pi)^2 ( \vec{ l } \cdot \vec{ k }) \hat{ a }_{ \vec{ k } - \vec{ l } } \hat{ v }_\vec{ l }\\
			&= \sum_{ \vec{l} \in (\{ \vec{ k } \} - \supp(\hat{ a })) \cap (\mathcal{S}^n \setminus \mathcal{S}^{ n - 1 })} (2\pi)^2 ( \vec{ l } \cdot \vec{ k }) \hat{ a }_{ \vec{ k } - \vec{ l } } \hat{ v }_\vec{ l }.
	\end{align*}

	This sum is nonempty only if $ \vec{ k } $ is such that there exists $ \vec{ l } \in \mathcal{S}^n \setminus \mathcal{S}^{ n - 1 } $ and $ \vec{ k }_{ a }^* \in \supp(\hat{ a }) $ with $ \vec{ k } = \vec{ l } + \vec{ k }_{ a }^*$.
	By definition of $ \vec{ l } \in \mathcal{S}^n \setminus \mathcal{ S }^{ n - 1 } $, $ n $ is the minimal such number that
	\begin{equation*}
		\vec{ l } = \vec{ k }_{ f } + \sum_{ m = 1 }^n \vec{ k }_{ a }^m, \text{ where } \vec{ k }_{ f } \in \supp(\hat{ f }), \; \vec{ k }_{ a }^m \in \supp(\hat{ a }) \text{ for all } m = 1, \ldots, n
	\end{equation*}
	holds.
	In particular, this implies that $ \vec{ k }_{ a }^m \neq \vec{ 0 } $ for all $ m = 1, \ldots, n $.

	There are now two cases. 
	First, if $ \vec{ k }_{ a }^* = - \vec{ k }_{ a }^m $ for any $ m $, $ \vec{ k } = \vec{ l } + \vec{ k }_{ a }^* \in \mathcal{S}^{ n - 1 } \setminus \mathcal{ S }^{ n - 2 } $, and the proposition is satisfied.
	On the other hand, we consider the case when $ \vec{ k }_{ a }^* $ does not negate any $ \vec{ k }_{ a }^m $ involved in the sum equalling $ \vec{ l } $.
	If $ \vec{ k }_{ a }^* = \vec{ 0 } $, then clearly $ \vec{ k } = \vec{ l } \in \mathcal{S}^n \setminus \mathcal{S}^{ n - 1 } $.
	In any other case, we represent 
	\begin{equation*}
		\vec{ k } = \vec{ k }_{ f } + \sum_{ m = 1 }^n \vec{ k }_{ a }^m + \vec{ k }_{ a }^* =: \vec{ k }_{ f } + \sum_{ m = 1 }^{ n + 1 } \vec{ k }_{ a }^m,
	\end{equation*}
	where $ n + 1 $ is the smallest number for which this holds.
	Thus, $ \vec{ k } \in \mathcal{S}^{ n + 1 } \setminus \mathcal{S}^n $.
	Altogether then, the only possible $ \vec{ k } $ values such that the sum is nonzero are those in $ \mathcal{S}^{ n + 1 } \setminus \mathcal{S}^{ n - 2 } $, completing the proof.
\end{proof}

Noting that $ \supp(L[\hat{ a }] \hat{ u }) = \supp(\hat{ f }) $, we observe the following interesting relationship between the values of $ \hat{ u } $ on neighboring stamping levels.
Below, to simplify notation, for all $ m, n \in \N_0 $, we set
\begin{equation*}
	b_{ m, n } := \langle L[\hat{ a }] \hat{ u }_{ \mathcal{S}^m\setminus \mathcal{S}^{ m - 1 } }, \hat{ u }_{ \mathcal{S}^n \setminus \mathcal{S}^{ n - 1 } } \rangle_{ \ell^2 }, 
\end{equation*}
with the convention  that $ \mathcal{S}^{ -1 } = \emptyset $.

\begin{cor}
	For all $ n \in \N_0 $,
	\begin{equation*}
		b_{ n + 1, n } + b_{ n, n } + b_{ n - 1, n } = 
			\begin{cases}
				\langle \hat{ f }, \hat{ u }\restrict{ \mathcal{S}^0 } \rangle_{ \ell^2 } &\text{ if } n = 0 \\
				0 &\text{ otherwise}.
			\end{cases}
	\end{equation*}
\end{cor}
\begin{proof}
	By Proposition~\ref{prop:LExpansion}, $ \hat{ u }\restrict{ \mathcal{S}^{ n } \setminus \mathcal{S}^{ n - 1 } } $ is $ \ell^2 $-orthogonal to $ L[\hat{ a }] \hat{ u }\restrict{ \mathcal{S}^m \setminus \mathcal{S}^{ m - 1 } } $ for all $ m \notin \{n - 1, n, n + 1\} $.
	In our simplified notation, $ b_{ m, n } = 0 $ for all $ m \notin \{n - 1, n, n + 1\} $.
	Thus
	\begin{equation*}
		\langle \hat{ f }, \hat{ u }\restrict{ \mathcal{S}^n \setminus \mathcal{S}^{ n - 1 } } \rangle_{ \ell^2 }
			= \langle L[\hat{ a }] \hat{ u }, \hat{ u }\restrict{ \mathcal{S}^n \setminus \mathcal{S}^{ n - 1 } } \rangle_{ \ell^2 }
			= \sum_{ m = 0 }^{ \infty } b_{ m, n }
			= b_{ n + 1, n } + b_{ n, n } + b_{ n - 1, n }.
	\end{equation*}
	The proof is finished by noting that
	\begin{equation*}
		\langle \hat{ f }, \hat{ u }\restrict{ \mathcal{S}^n \setminus \mathcal{S}^{ n - 1 } } \rangle_{ \ell^2 }
			=
			\begin{cases}
				\langle \hat{ f }, \hat{ u }\restrict{ \mathcal{S}^0 } \rangle &\text{ if } n = 0 \\
				0 &\text{ otherwise}.
			\end{cases}
	\end{equation*}
\end{proof}

We are now ready to estimate $ \hat{ u }\restrict{ \mathcal{S}^n \setminus \mathcal{S}^{ n - 1 } } $ in terms of its neighbors $ \hat{ u }\restrict{ \mathcal{S}^{ n + 1 } \setminus \mathcal{S}^{ n } } $ and $ \hat{ u }\restrict{ \mathcal{S}^{ n - 1 } \setminus \mathcal{S}^{ n - 2 } } $.
The standard approach would be to use a combination of coercivity and continuity (see, e.g., the proof of Lemma~\ref{lem:StrangsLemma} or \cite[Section~6.4]{canuto_spectral_2006} for other examples): for $ n > 0 $,
\begin{equation*}
	\alpha \norm{ u\restrict{ \mathcal{S}^n \setminus \mathcal{S}^{ n - 1 } } }_{ H }^2 \leq |b_{ n, n }| \leq | b_{ n + 1, n }| + |b_{ n - 1, n }| \leq \beta \norm{ u\restrict{ \mathcal{S}^n \setminus \mathcal{S}^{ n - 1 } } }_{ H } \left( \norm{ u\restrict{ \mathcal{S}^{ n + 1 } \setminus \mathcal{S}^{ n}  } }_{ H }  + \norm{ u\restrict{ \mathcal{S}^{ n - 1 } \setminus \mathcal{S}^{ n - 2 } } }_{ H }\right),
\end{equation*}
and we obtain
\begin{equation*}
	\norm{ u\restrict{ \mathcal{S}^n \setminus \mathcal{S}^{ n - 1 } } }_{ H } \leq  \frac{\beta}{\alpha} \left( \norm{ u\restrict{ \mathcal{S}^{ n + 1 } \setminus \mathcal{S}^{ n}  } }_{ H }  + \norm{ u\restrict{ \mathcal{S}^{ n - 1 } \setminus \mathcal{S}^{ n - 2 } } }_{ H }\right).
\end{equation*}
However, we will hope to iterate this bound, and the fact that $ \beta \geq \alpha $ will not allow for us to show any decay as $ n \rightarrow \infty $.
Thus, we require a slightly subtler estimate than simply using continuity.

\begin{prop}
	\label{prop:RefinedContinuity}
	For $ n > 0 $, we have 
	\begin{equation*}
		|b_{ n \pm 1, n }| \leq \norm{ a - \hat{ a }_\vec{ 0 } }_{ L^\infty } \norm{ u\restrict{ \mathcal{S}^n \setminus \mathcal{S}^{ n - 1 }} }_{ H }\norm{ u\restrict{ \mathcal{S}^{ n \pm 1 } \setminus \mathcal{S}^{ n \pm 1 - 1 }} }_{ H }.
	\end{equation*}
\end{prop}
\begin{proof}
	Restricting all sums to the support of the vectors they index, we have
	\begin{equation*}
		b_{ n \pm 1, n } = \sum_{ \vec{ k } \in \mathcal{S}^{ n } \setminus \mathcal{S}^{ n - 1 } } \sum_{ \vec{ l } \in (\vec{ k } - \supp(\hat{ a })) \cap (\mathcal{S}^{ n \pm 1 } \setminus \mathcal{S}^{ n \pm 1 - 1 }) } (2 \pi)^2 (\vec{ l } \cdot \vec{ k }) \hat{ a }_{ \vec{ k } - \vec{ l } } \hat{ u }_{ \vec{ l } } \overline{\hat{ u }}_{ \vec{ k } }.
	\end{equation*}
	Clearly, choosing $ \vec{ l } = \vec{ k } \in \mathcal{S}^{ n } \setminus \mathcal{S}^{ n - 1  } $ would not allow for $ \vec{ l } \in \mathcal{S}^{ n \pm 1 } \setminus \mathcal{S}^{ n \pm 1 - 1 } $.
	Thus, no term multiplying $ \hat{ a }_{ \vec{ k } - \vec{ k } } = \hat{ a }_{ \vec{ 0 } } $ will appear in this sum.
	We then have the equivalence
	\begin{equation*}
		b_{ n \pm 1, n } = \langle L[\hat{ a } - \hat{ a }_\vec{ 0 }] \hat{ u }\restrict{ \mathcal{S}^{ n \pm 1 } \setminus \mathcal{S}^{n \pm 1 - 1} }, \hat{ u }\restrict{ \mathcal{S}^n \setminus \mathcal{S}^{ n - 1 } } \rangle_{ \ell^2 },
	\end{equation*}
	which by the standard argument for continuity, implies
	\begin{equation*}
		|b_{ n \pm 1, n }| \leq \norm{ a - \hat{ a }_\vec{ 0 } }_{ L^\infty } \norm{ u\restrict{ \mathcal{S}^n \setminus \mathcal{S}^{ n - 1 }} }_{ H }\norm{ u\restrict{ \mathcal{S}^{ n \pm 1 } \setminus \mathcal{S}^{ n \pm 1 - 1 }} }_{ H }.
	\end{equation*}
	as desired.
\end{proof}

The same argument preceding Proposition~\ref{prop:RefinedContinuity} then gives the desired ``neighbor'' estimate.
\begin{cor}
	\label{cor:NeighborBound}
	For all $ n > 1 $,
	\begin{equation*}
	\norm{ u\restrict{ \mathcal{S}^n \setminus \mathcal{S}^{ n - 1 } } }_{ H } \leq  \frac{\norm{ a - \hat{ a }_\vec{ 0 } }_{ L^\infty }}{a_\mathrm{min}} \left( \norm{ u\restrict{ \mathcal{S}^{ n + 1 } \setminus \mathcal{S}^{ n}  } }_{ H }  + \norm{ u\restrict{ \mathcal{S}^{ n - 1 } \setminus \mathcal{S}^{ n - 2 } } }_{ H }\right).
	\end{equation*}
\end{cor}

We now have the pieces to state an estimate of the truncation error.
\begin{lem}
	\label{lem:StampDecay}
	Let $ a $, $ f $, and $ u $ be as in Proposition~\ref{prop:ExistenceUniquenessStability}.
	Assume
	\begin{equation}
		\label{eq:aCondition}
		3 \norm{ a - \hat{ a }_\vec{ 0 } }_{ L^\infty } < a_\mathrm{min}
	\end{equation}
	Then
	\begin{equation*}
		\norm{ u - u\restrict{ \mathcal{S}^N}}_{ H } \leq \left( \frac{ \norm{ a - \hat{ a }_\vec{ 0 } }_{ L^\infty } }{a_\mathrm{min} - 2 \norm{ a - \hat{ a }_\vec{ 0 } }_{ L^\infty }} \right)^{ N + 1 } \frac{\norm{ f }_{ L^2 }}{a_\mathrm{min}}.
	\end{equation*}
\end{lem}
\begin{proof}
	We begin by breaking $ \supp(\hat{ u }) \setminus \mathcal{S}^N $ into sets of new contributions $ \bigcup_{ n = N + 1 }^\infty \left(\mathcal{S}^{ n } \setminus \mathcal{S}^{ n - 1 }\right) $ (which holds due to Proposition~\ref{prop:SolutionSupport}).
	Thus
	\begin{equation*}
		\norm{ u - u\restrict{ \mathcal{S}^N } }_{ H } \leq \sum_{ n = N + 1 }^\infty \norm{ u\restrict{ \mathcal{S}^n \setminus \mathcal{S}^{ n - 1 } } }_{ H } =: T_N.
	\end{equation*}
	Applying the neighbor bound, Corollary~\ref{cor:NeighborBound}, (where we define $ A := \norm{ a - \hat{ a }_\vec{ 0 } }_{ L^\infty } / a_\mathrm{min} $), we have
	\begin{align*}
		T_N 
			&\leq A \left( \sum_{ n = N + 1 }^{ \infty }\norm{ u\restrict{ \mathcal{S}^{ n + 1 } \setminus \mathcal{S}^n } }_{ H } + \sum_{ n = N + 1 }^\infty \norm{ u\restrict{ \mathcal{S}^{ n - 1 } \setminus \mathcal{S}^{ n - 2 } } }_{ H } \right)\\
			&= A \left( T_{ N + 1 } + T_{ N - 1 } \right) \\
			&= 2 A T_N  + A\left( \norm{ u\restrict{ \mathcal{S}^{ N } \setminus \mathcal{S}^{N - 1} } }_{ H } - \norm{ u\restrict{ \mathcal{S}^{ N + 1 } \setminus \mathcal{S}^{ N } } }_{ H } \right).
	\end{align*}
	After rearranging, and ignoring the negative term, we find
	\begin{equation}
		\label{eq:SumByPiece}
			T_N \leq \frac{A}{1 - 2A} \norm{ u\restrict{ \mathcal{S}^N \setminus \mathcal{S}^{ N - 1 } } }_{ H }.
	\end{equation}
	Noting that we always have
	\begin{equation}
		\label{eq:PieceBySum}
		\norm{ u\restrict{ \mathcal{S}^N \setminus \mathcal{S}^{ N - 1 } } }_{ H } \leq T_{ N - 1 },
	\end{equation}
	iterating \eqref{eq:SumByPiece} and \eqref{eq:PieceBySum} in turn gives
	\begin{equation*}
		\norm{ u - u\restrict{ \mathcal{S}^N } }_{ H } \leq T_N \leq \left( \frac{A}{1 - 2A}  \right)^{ N + 1 } \norm{ u\restrict{ \mathcal{S}^0 } }_{ H } \leq \left( \frac{A}{1 - 2A}  \right)^{ N + 1 } \frac{\norm{ f }_{ L^2 }}{a_\mathrm{min}}.
	\end{equation*}
\end{proof}

\section{Previous results on SFTs}\label{sec:previous_results}

In \cite{gross_sparse_2021}, two methods for high-dimensional SFTs are presented, each with a deterministic and Monte Carlo variant.
Here, we use the faster of the two algorithms (at the cost of slightly suboptimal error guarantees).
We focus on only the Monte Carlo variant as the improvements to this technique described in Section~\ref{sec:improvements_with_randomized_lattices} below use an additional layer of randomization.

This method relies on applying one-dimensional SFTs to samples of a high-dimensional function along special sets called \emph{reconstructing rank-1 lattices}.

\begin{defn}
	Given a number of sampling points $ M \in \N $ and a generating vector $ \vec{ z } \in \{1, \ldots M - 1 \}^d $, we define the \emph{rank-1 lattice} $ \Lambda(\vec{ z }, M) $ as the set
	\begin{equation*}
		\Lambda(\vec{ z }, M) := \left\{ \frac{j}{M} \vec{z} \bmod \vec{ 1 } \mid j \in \{0, \ldots, M - 1 \} \right\} \subset \T^d.
	\end{equation*}
	Additionally, given a set of frequencies $ \mathcal{I} \subset \Z^d $, we say that $ \Lambda(\vec{ z }, M) $ is a \emph{reconstructing rank-1 lattice for} $ \mathcal{I} $ if
	\begin{equation*}
		\vec{ l } \cdot \vec{ z } \not \equiv \vec{ k } \cdot \vec{ z } \bmod M \quad \text{for all } \vec{ l } \neq \vec{ k } \in \mathcal{I}.
	\end{equation*}
\end{defn}

The fundamental idea of a reconstructing rank-1 lattice is that it takes a multivariate function $ g: \T^d \rightarrow \R $ and gives the locations for $ M $ equispaced samples of the univariate function $ t \mapsto g(t \vec{ z }) $.
The univariate Fourier content of these samples can then be assigned to the original function $ g $ with the reconstructing property ensuring that no multidimensional frequencies of interest are aliased together in the one-dimensional analysis.  
For the following theorem, we assume that we know a reconstructing rank-1 lattice exists for a given frequency set of interest, $ \mathcal{I} $.
This assumption will be lifted in the following section.

The following theorem is a restatement of \cite[Corollary 2]{gross_sparse_2021} with minor simplifications and improvements (most notably, $ L^\infty $ error bounds).
The proof of these improvements is given in Appendix~\ref{sec:proof_of_sft_recovery_guarantees}.

\begin{thm}[\cite{gross_sparse_2021}, Corollary 2]
	\label{thm:SFTRecovery}
	Let $ \mathcal{I} \subset \Z^d$ be a frequency set of interest with expansion defined as $ K := \max_{ j \in \{1, \ldots, d\} } (\max_{ \vec{ k } \in \mathcal{I} } k_j - \min_{ \vec{ l } \in \mathcal{I}} l_j ) $ (i.e., the sidelength of the smallest hypercube containing $ \mathcal{I} $), and $ \Lambda( \vec{ z }, M) $ be a reconstructing rank-1 lattice for $ \mathcal{I} $.

	There exists a fast, randomized SFT which, given $ \Lambda(\vec{ z }, M) $, sampling access to $ g \in L^2 $, and a failure probability $ \sigma \in (0, 1] $, will produce a $ 2s $-sparse approximation $ \hat{ \vec{ g } }^s $ of $ \hat g $ and function $ g^s := \sum_{ \vec{ k } \in \supp( \hat{ \vec{ g }}^s) } \hat g_{ \vec{ k } }^s e_\vec{ k } $ approximating $ g $ satisfying
	\begin{align*}
		\norm{ g - g^s }_{ L^2 } \leq \norm{ \hat g - \hat{ \vec{ g } }^s }_{ \ell^2 } 
			&\leq (25 + 3K) \left[ \frac{\norm{\hat{ g }\restrict{ \mathcal{I} } - (\hat{ g }\restrict{ \mathcal{I} })_s^\mathrm{opt}}_1}{\sqrt{ s }} + \sqrt{ s } \norm{ \hat{ g } - \hat{ g }\restrict{ \mathcal{I} } }_1 \right]
	\end{align*}
	with probability exceeding $ 1 - \sigma $.
	If $ g \in L^\infty $, then we additionally have
	\begin{equation*}
		\norm{ g - g^s }_{ L^\infty } \leq \norm{ \hat g - \hat{ \vec{ g } }^s }_{ \ell^1 } \leq (33 + 4K) \left[ \norm{ \hat{ g } \restrict{ \mathcal{I} } - (\hat{ g }\restrict{ \mathcal{I} })_s^\mathrm{opt} }_1 + \norm{ \hat{ g } - \hat{ g }\restrict{ \mathcal{I} } }_1 \right]
	\end{equation*}
	with the same probability estimate.
	The total number of samples of $ g $ and computational complexity of the algorithm can be bounded above by
	\begin{equation*}
		\mathcal{O} \left( d s \log^3( d K M) \log \left( \frac{d K M }{\sigma} \right)  \right).
	\end{equation*}
\end{thm}

\section{Improvements with randomized lattices}\label{sec:improvements_with_randomized_lattices}

To use the previous SFT algorithm, we need to know a reconstructing rank-1 lattice in advance.
Though there are deterministic algorithms to construct a reconstructing rank-1 lattice given any frequency set $ \mathcal{I} $ (for example, the component-by-component construction \cite{plonka_numerical_2018, kuo_function_2021}), these algorithms are are superlinear in $ | \mathcal{I} | $ as they effectively search the frequency space for collisions throughout construction.

This section presents an alternative based on choosing a random lattice.
This lattice is chosen by drawing $ \vec{ z } $ from a uniform distribution over $ \{1, \ldots, M - 1\}^d $  for $ M $  sufficiently large.
Below, we provide probability estimates for when this lattice is reconstructing for a frequency set $ \mathcal{I} $.

\begin{lem}
	\label{lem:RandomRank1Lattice}
	Let $ K := \max_{ j \in \{1, \ldots d\} } (\max_{ \vec{ k } \in \mathcal{I}} k_j - \min_{ \vec{ l } \in \mathcal{I} } l_j)$ be the expansion of the frequency set $ \mathcal{I} \subset \Z^d $.
	Let $ \sigma \in (0, 1] $, and fix $ M $ to be the smallest prime greater than $ \max(K, \frac{|\mathcal{I}|^2}{\sigma})  $.
	Then drawing each component of $ \vec{ z } $ i.i.d from $ \{1, \ldots M-1\} $ gives that $ \Lambda(\vec{ z }, M) $ is a reconstructing rank-1 lattice for $ \mathcal{I} $ with probability $ 1 - \sigma$.
\end{lem}
\begin{proof}
	In order to show that $ \Lambda(\vec{ z }, M) $ is reconstructing for $ \mathcal{I} $, it suffices to show that for any $ \vec{ k } \neq \vec{ l } \in \mathcal{I} $, $ \vec{ k } \cdot \vec{ z } \not\equiv \vec{ l } \cdot \vec{ z } \bmod M $.
	Thus, we are interested in showing that $ \P[ \exists \vec{ k } \neq \vec{ l } \in \mathcal{I} \text{ s.t. } (\vec{ k } - \vec{ l }) \cdot \vec{ z } \equiv \vec{ 0 } \bmod M] $ is small.

	If $ \vec{ k }, \vec{ l } \in \mathcal{I} $ are distinct, at least one component $ k_j - l_j $ is nonzero.
	Since $ M > K $, we therefore have that $ k_j - l_j \not\equiv 0 \bmod M $, and since $ M $ is prime, $ k_j - l_j $ has a multiplicative inverse modulo $ M $.
	Then $ \P[(\vec{ k } - \vec{ l }) \cdot \vec{ z }\equiv \vec{ 0 } \bmod M] = \P\left[z_j = \left( (k_j - l_j)^{ -1 } \sum_{ i \in \{1, \ldots d\}, i \neq j } (k_i - l_i)z_i \bmod M \right)\right] $.
	Since $ z_j $ is uniformly distributed in $ \{1, \ldots M - 1\} $, this probability is $ \frac{1}{M - 1} $.
	By the union bound,
	\begin{equation*}
		\P[ \exists \vec{ k } \neq \vec{ l } \in \mathcal{I} \text{ s.t. } (\vec{ k } - \vec{ l }) \cdot \vec{ z } \equiv \vec{ 0 } \bmod M] \leq \sum_{ \vec{ k } \neq \vec{ l } \in \mathcal{I} } \P[(\vec{ k } - \vec{ l }) \cdot \vec{ z } \equiv \vec{ 0 } \bmod M] \leq \frac{\abs{ \mathcal{I} }^2}{M - 1} \leq \sigma
	\end{equation*}
	as desired.

\end{proof}

One important consequence of Lemma~\ref{lem:RandomRank1Lattice} is that we no longer need to provide the frequency set of interest in Theorem~\ref{thm:SFTRecovery}.
Having chosen $ K $, the expansion, and $ s $, the sparsity level, we can always take $ \mathcal{I} $ to be the frequencies corresponding to the largest $ s $ Fourier coefficients of the function $ g $ in the hypercube $ [-K/2, K/2]^d $.
Lemma~\ref{lem:RandomRank1Lattice} then implies that a randomly generated lattice with length $ \max(K, s^2 / \sigma) $ will be reconstructing for these optimal frequencies with probability $ \sigma $.
We summarize this in the following corollary.

\begin{cor}
	\label{cor:SFT}
	For a multivariate function's Fourier series $ \hat{ g } $, define $ \hat{ g }\restrict{ K } := \hat{ g }\restrict{ [-K/2, K/2]^d } $.
	Given a multivariate bandwidth $ K $, a sparsity level $ s $, probability of failure $ \sigma \in (0, 1] $, and sampling access to $ g \in L^2 $, there exists a fast, randomized SFT which will produce a $ 2s $-sparse approximation $ \hat{ \vec{ g } }^s $ of $ \hat{ g } $ and function $ g^s := \sum_{ \vec{ k } \in \supp(\hat{ \vec{ g } }^s) } \hat{ g }_\vec{ k }^s e_\vec{ k } $ approximating $ g $ satisfying
	\begin{equation*}
		\norm{ g - g^s }_{ L^2 } \leq \norm{ \hat{ g } - \hat{ \vec{ g } }^s }_{ \ell^2 } \leq (25 + 3K) \sqrt{ s } \norm{ \hat{ g } - (\hat{ g }\restrict{ K })_s^\mathrm{opt} }_{ \ell^1 }
	\end{equation*}
	with probability $ 1 - \sigma $.
	If $ g \in L^\infty $, then $ g^s $ and $ \hat{ g }^s $ satisfy the upper bound
	\begin{equation*}
		\norm{ g - g^s }_{ L^\infty } \leq \norm{ \hat{ g } - \hat{ \vec{ g } }^s }_{ \ell^1 } \leq (33 + 4K) \norm{ \hat{ g } - (\hat{ g }\restrict{ K })_s^\mathrm{ opt} }_{ \ell^1 }
	\end{equation*}
	with the same probability estimate.
	The total number of samples of $ g $ and computational complexity of the algorithm can be bounded above by
	\begin{equation*}
	\mathcal{O}\left(d s \log^3(d K \max(K, s / \sigma)) \log \left( \frac{d K \max(K, s / \sigma)}{\sigma}  \right)\right).
	\end{equation*}
	If we fix $ \sigma $ (say $ \sigma = 0.95 $), this reduces to a complexity of
	\begin{equation*}
		\mathcal{O}\left(d s \log^4(d K \max(K, s))\right).
	\end{equation*}
\end{cor}

\section{A sparse spectral method via SFTs}\label{sec:in_the_language_of_canuto_spectral_2006}

Let $ \hat{ \vec{ a } }^s $ and $ \hat{ \vec{ f } }^s $ be $ s $-sparse approximations of $ \hat{ a } $ and $ \hat{ f } $ respectively.
We will use these approximations to discretize the Galerkin formulation \eqref{eq:GalerkinForm} of our PDE.
The first step is to reduce to the case where the PDE data is Fourier-sparse which is motivated by the following lemma.

\begin{lem}
	\label{lem:RestrictedSupportRecovery}
	Let $ a' := a\restrict{ \supp \hat{ \vec{ a } }^s  } $ and $ f' := f\restrict{ \supp \hat{ \vec{ f } }^s } $.
	Suppose that $ a' $ and $ f' $ satisfy the conditions of Proposition~\ref{prop:ExistenceUniquenessStability} and let $ u' $ be the unique solution of the resulting elliptic PDE, which we write in Galerkin form as
	\begin{equation}
		\label{eq:GalerkinFormSupportApproximation1}
		L[\hat a'] \hat u' = \hat f'.
	\end{equation}
	Then
	\begin{equation*}
		\norm{ u - u' }_{ H } \leq \frac{\norm{ f - f' }_{ L^2 }}{a_\mathrm{ min }} + \frac{\norm{ a - a' }_{ L^\infty } \norm{ f' }_{ L^2 }}{a_\mathrm{min} a'_\mathrm{min}}.
	\end{equation*}
\end{lem}
\begin{proof}
	We begin by observing
	\begin{equation*}
		L[\hat a](\hat u - \hat u') = L[\hat a] \hat u - L[\hat a'] \hat u' - L[\hat a - \hat a'] \hat u' = \hat f - \hat f' - L[\hat a - \hat a'] \hat u',
	\end{equation*}
	and therefore
	\begin{equation*}
		\abs{ \langle L[\hat a] (\hat u - \hat u'), \hat u - \hat u' \rangle } \leq \abs{ \langle \hat f - \hat f', \hat u - \hat u' \rangle } + \abs{ \langle L[\hat a - \hat a'] \hat u', \hat u - \hat u' \rangle }.
	\end{equation*}
	After an application of Proposition~\ref{prop:GalerkinBilinearForm} to convert the $ \ell^2 $ inner products into bilinear forms, we can make use of coercivity, \eqref{eq:Coercivity}, continuity, \eqref{eq:Continuity} and the Cauchy-Schwarz inequality to produce the $ H $ approximation
	\begin{equation*}
		a_\mathrm{ min } \norm{ u - u' }_{ H } \leq \norm{ \hat f - \hat f' }_{ \ell^2 } + \norm{ a - a' }_{ L^\infty } \norm{ u' }_{ H }.
	\end{equation*}
	An application of the stability estimate \eqref{eq:StabilityEstimate} gives the desired bound
	\begin{equation*}
		\norm{ u - u' }_{ H } \leq \frac{\norm{ f - f' }_{ L^2 }}{a_\mathrm{ min }} + \frac{\norm{ a - a' }_{ L^\infty } \norm{ f' }_{ L^2 }}{a_\mathrm{min} a'_\mathrm{min}}.
	\end{equation*}
\end{proof}

We can now replace the trial and test spaces in \eqref{eq:WeakPDEWholeTestSpace} with finite dimensional approximations so as to convert \eqref{eq:GalerkinForm} to a matrix equation.
Inspired by Proposition~\ref{prop:SolutionSupport} and the truncation error analysis in Section~\ref{sec:stamping_sets}, we use the space of functions whose Fourier coefficients are supported on $ \mathcal{S}^N := \mathcal{S}^N[ \hat{ a }](\supp \hat{ f } )$.
By doing so, we discretize the Galerkin formulation of the problem \eqref{eq:GalerkinForm} into the finite system of equations
\begin{equation}
	\label{eq:FiniteGalerkinOperator}
	(\vec{ L }_N \hat{ \vec{ u } })_{ \vec{k} } := \sum_{ \vec{l} \in \mathcal{S}^N } (2\pi)^2 ( \vec{ l } \cdot \vec{ k } ) \hat{ a }_{ \vec{ k } - \vec{ l } } \hat{ u }_{ \vec{ l } } = \hat{ f }_\vec{ k }  \quad \text{ for all } \vec{ k } \in \mathcal{S}^N.
\end{equation}
However, in practice, we do not know $ \hat{ a } $ and $ \hat{ f } $ exactly (and indeed, they may not be exactly sparse).
Thus, we substitute the SFT approximations $ \hat{ \vec{ a } }^s $ and $ \hat{ \vec{ f } }^s $, defining the new finite-dimensional operator $ \vec{ L }_{ N, s }: \C^{ \mathcal{S}^N } \rightarrow \C^{ \mathcal{S}^N } $ by
\begin{equation*}
	\left( \vec{L}_{ N, s } \hat{ \vec{ u } } \right)_\vec{ k } := \sum_{ \vec{ l } \in \mathcal{S}^N } (2\pi)^2 (\vec{ l } \cdot \vec{ k }) \hat{ a }_{ \vec{ k } - \vec{ l } }^s \hat{ u }_\vec{ l } \quad \text{ for all } \vec{ k } \in \mathcal{S}^N.
\end{equation*}
Our new approximate solution will be $ \hat{ \vec{ u } }^{ N, s } \in \C^{ \mathcal{S}^N } $ which solves
\begin{equation}
	\label{eq:SparseSpectralSFTEquation}
	\vec{ L }_{ N, s } \hat{ \vec{ u } }^{ N, s } = \hat{ \vec{ f } }^s.
\end{equation}
We summarize our technique in Algorithm~\ref{alg:SparseSFT}.

\begin{algorithm}[H]
	\caption{Sparse spectral method}
	\label{alg:SparseSFT}
	\begin{algorithmic}[1]
		\Require PDE data $ a $ and $ f $, a sparsity parameter $ s $, a bandwidth parameter $ K $, and a stamping level $ N $
		\Ensure Fourier coefficients $ \hat{ \vec{ u } }^{ s, N } $ of approximate solution
		\State$ \hat{ \vec{ a } }^s \gets \mathrm{SFT}[s, K](a) $ \hfill \Comment{$ \mathrm{SFT} $ is the algorithm in \cite{gross_sparse_2021} using a random rank-1 lattice (cf.\ Section~\ref{sec:improvements_with_randomized_lattices})}
		\State $ \hat{ \vec{ f } }^s \gets \mathrm{SFT}[s, K](f) $
		\State Compute $ \mathcal{ S }^N[\hat{ \vec{ a } }^s](\supp(\hat{ \vec{ f } }^s)) $ \hfill \Comment{see, e.g., \eqref{eq:StampSet} or \eqref{eq:StampFrequencyEnumeration}}
		\State $ (\vec{ L }_{ N, s })_{ \vec{ k } \in \mathcal{S}^N, \vec{ l } \in \mathcal{S}^N} \gets (2\pi)^2 (\vec{ l } \cdot \vec{ k }) \hat{ a }^s_{ \vec{ k } - \vec{ l } } $
		\State $ \hat{ \vec{ u } }^{ N, s } \gets \vec{ L }_{ N, s } \backslash \hat{ \vec{ f } }^s $ \hfill \Comment{using MATLAB backslash notation for matrix solve}
 	\end{algorithmic}
\end{algorithm}

Showing that $ u^{ N, s } $ converges to $ u $ now relies on a version of Strang's lemma \cite[Equation (6.4.46)]{canuto_spectral_2006}.
We make the assumption here that $ \supp(\hat{ a }) = \supp(\hat{ \vec{ a } }^s) $ and $ \supp(\hat{ f }) = \supp(\hat{ \vec{ f } }^s) $ so that our use of $ \mathcal{S}^N $ is unambiguous.
However, this assumption will be lifted by Lemma~\ref{lem:RestrictedSupportRecovery} in Corollary~\ref{cor:SpectralConvergenceNoSFT} below.
\begin{lem}[Strang's Lemma]
	\label{lem:StrangsLemma}
	Suppose that $ \supp(\hat{ a }) = \supp(\hat{ \vec{ a } }^s) $ and $ \supp(\hat{ f }) = \supp(\hat{ \vec{ f } }^s) $.
	Also suppose that $ a^s \geq a^s_\mathrm{min} > 0 $ on $ \T^d $.
	Let $ u $ and $ u^{ N, s } $ be as above.
	Then
	\begin{equation*}
		\norm{ u - u^{ N, s } }_{ H } \leq \left( 1 + \frac{ \norm{ a }_{ L^\infty } }{ a^s_\mathrm{min}}  \right) \norm{ u\restrict{ \Z^d \setminus \mathcal{S}^N } }_{ H } + \frac{\norm{ a - a^s }_{L^\infty}}{a^s_\mathrm{min}} \norm{ u\restrict{ \mathcal{S}^N } }_{ H } + \frac{\norm{ f - f^s }_{ L^2 }}{a^s_\mathrm{min}}.
	\end{equation*}
\end{lem}
\begin{proof}
	We let $ \hat{ \vec{ e } } := \hat{ \vec{ u } }^{ N, s } - \hat{ u }\restrict{ \mathcal{S}^N } $, and consider
	\begin{align*}
		\vec{ L }_{ N, s } \hat{ \vec{ e } }
			&= \vec{ L }_{ N, s } \hat{ \vec{ u } }^{ N, s } - (L[\hat{ \vec{ a } }^s] \hat{ u }\restrict{ \mathcal{S}^N })\restrict{ \mathcal{S}^N }\\
			&= \hat{ \vec{ f } }^s - \hat{ f } + (L[\hat{ a }] \hat{ u })\restrict{ \mathcal{S}^N } - (L[\hat{ \vec{ a } }^s] \hat{ u }\restrict{ \mathcal{S}^N })\restrict{ \mathcal{S}^N }\\
			&= \hat{ \vec{ f } }^s - \hat{ f } + (L[\hat{ a }] \hat{ u }\restrict{ \Z^d \setminus \mathcal{S}^N })\restrict{ \mathcal{S}^N } + (L[\hat{ a }] \hat{ u }\restrict{ \mathcal{S}^N } - L[\hat{ \vec{ a } }^s] \hat{ u }\restrict{ \mathcal{S}^N })\restrict{ \mathcal{S}^N }\\
			&= \hat{ \vec{ f } }^s - \hat{ f } + (L[\hat{ a }] \hat{ u }\restrict{ \Z^d \setminus \mathcal{S}^N })\restrict{ \mathcal{S}^N } + (L[\hat{ a } - \hat{ \vec{ a } }^s] \hat{ u }\restrict{ \mathcal{S}^N } )\restrict{ \mathcal{S}^N }.
	\end{align*}
	Noting that $ \vec{ L }_{ N, s } \hat{ \vec{ e } } = (L[\hat{ \vec{ a } }^s] \hat{ \vec{ e } })\restrict{ \mathcal{S}^N } $ and owing to coercivity of $ L[\hat{ \vec{ a } }^s] $, we have
	\begin{align*}
		a^s_\mathrm{min} \norm{ e }_{ H }^2 
			&\leq \abs{ \langle \vec{ L }_{ N, s } \hat{ \vec{ e } }, \hat{ \vec{ e } } \rangle } \\
			&\leq \norm{ f^s - f }_{ L^2 } \norm{ e }_{ H } + \norm{ a }_{ L^\infty } \norm{  u\restrict{ \Z^d \setminus \mathcal{S}^N } }_{ H } \norm{ e }_{ H } + \norm{ a - a^s }_{ L^\infty } \norm{ u\restrict{ \mathcal{S}^N } }_{ H } \norm{ e }_{ H }.
	\end{align*}
	The result then follows from rearranging to estimate $ \norm{ e }_{ H } $ and using the triangle inequality to estimate $ \norm{ u - u^{ N, s } }_{ H } \leq \norm{ u - u\restrict{ \mathcal{S}^N } }_{ H } + \norm{ e }_{ H } $.
\end{proof}

We can now thread all of our results together into a final convergence analysis.
The first corollary below is a more direct application of Strang's lemma which is then followed by another corollary which takes advantage of the SFT recovery results.
We will also return to the setting where $ a $ and $ f $ are not necessarily Fourier sparse.
Thus, for $ a^s $ and $ f^s $ Fourier sparse approximations of $ a $ and $ f $, we again let $ a' = a\restrict{ \supp \hat{ \vec{ a } }^s } $ and $ f' = f\restrict{ \supp \hat{ \vec{ f } }^s } $ as in Lemma~\ref{lem:RestrictedSupportRecovery}.

\begin{cor}
	\label{cor:SpectralConvergenceNoSFT}
	Suppose $ a $, $ f $ and $ a^s $, $ f^s $ respectively satisfy the conditions of Proposition~\ref{prop:ExistenceUniquenessStability}.
	Additionally, suppose that
	\begin{equation}
		\label{eq:IterationCondition}
		3 \sum_{ \vec{ k } \in \supp(\hat{ \vec{ a } }^s) \setminus \{\vec{ 0 }\} } \abs{\hat{ a }_\vec{ k }} \leq \hat{ a }_\vec{ 0 }.
	\end{equation}
	Then with $ u $ the exact solution to \eqref{eq:WeakPDEWholeTestSpace} and $ u^{ N, s } $ the output of Algorithm~\ref{alg:SparseSFT}, we have
	\begin{align*}
		\norm{ u - u^{ N, s } }_{ H } 
			&\leq \frac{\norm{ f - f' }_{ L^2 }}{a_\mathrm{min}} + \frac{\norm{ a - a' }_{ L^\infty }\norm{ f' }_{ L^2 }}{a_\mathrm{min}a_\mathrm{min}'}  + \left( 1 + \frac{\norm{ a' }_{ L^{ \infty } }}{a_\mathrm{min}^s}  \right) \left( \frac{\norm{ a' - \hat{ a }'_\vec{ 0 } }_{ L^\infty }}{a'_\mathrm{min} - 2 \norm{ a' - \hat{ a }'_\vec{0} }_{ L^\infty }}  \right)^{ N + 1 } \frac{\norm{ f' }_{ L^2 }}{a_\mathrm{min}'} \\
			&\qquad+ \frac{\norm{ a' - a^s }_{ L^\infty } \norm{ f' }_{ L^2 }}{a_\mathrm{min}^s a_\mathrm{min}} + \frac{\norm{ f' - f^s }_{ L^2 }}{a_\mathrm{min}^s}
	\end{align*}
\end{cor}
\begin{proof}
	The condition \eqref{eq:IterationCondition} ensures that $ a' $ is coercive, and therefore $ a' $ and $ f' $ also satisfy Proposition~\ref{prop:ExistenceUniquenessStability}.
	Additionally, this allows the use of Lemma~\ref{lem:StampDecay}, which upper bounds the truncation error in Lemma~\ref{lem:StrangsLemma}.
	Combining Lemma~\ref{lem:RestrictedSupportRecovery} with this bound from Lemma~\ref{lem:StrangsLemma} and applying the stability estimate from Proposition~\ref{prop:ExistenceUniquenessStability} finishes the proof.
\end{proof}

\begin{rem}
	In order for this bound to hold, it is necessary for the weak forms of both
	\begin{equation*}
		\mathcal{L}[a] u = f \text{ and } \mathcal{L}[a^s]u^s = f^s
	\end{equation*}
	to be well-posed, that is, satisfy the continuity and coercivity conditions of Proposition~\ref{prop:ExistenceUniquenessStability}.
	In practice, this condition is not much more restrictive than assuming only the original equation is well-posed as long as the diffusion coefficient is Fourier-compressible and the sparsity level $ s $ is large enough to ensure that $ a^s $ stays strictly positive.
	In fact, \eqref{eq:IterationCondition} allows for the simple (if pessimistic) check after computing $ \hat{ \vec{ a } }^s $ that $ \norm{ \hat{ \vec{ a } }^s - \hat{ a }^s_\vec{ 0 } }_{ \ell^1 } < \abs{ \hat{ a }^s_\vec{ 0 } } $ to ensure the positivity of $ a^s $.
\end{rem}

With minor modifications, we can rewrite this upper bound to pass all dependence on sparsity through the error in approximating $ a $ and $ f $ via SFTs.

\begin{cor}
	\label{cor:SpectralConvergenceWithSFT}
	Under the same conditions as Corollary~\ref{cor:SpectralConvergenceNoSFT} above substituting \eqref{eq:IterationCondition} with
	\begin{equation*}
		3 \norm{ \hat{ a } - \hat{ a }_\vec{ 0 } }_{ \ell^1 } + \norm{ \hat{ a } - \hat{ \vec{ a } }^s }_{ \ell^1 } < \hat{ a }_\vec{ 0 },
	\end{equation*}
	we have
	\begin{align*}
		\norm{ u - u^{ N, s } }_{ H } 
			&\leq \left( 1 + \frac{\norm{ \hat{ a } }_{ \ell^1 }}{a_\mathrm{min} - \norm{ \hat{ a } - \hat{ \vec{ a } }^s }_{ \ell^1 }}  \right) \frac{\norm{ f }_{ L^2 }}{a_\mathrm{min} - \norm{ \hat{ a } - \hat{ \vec{ a } }^s }_{ \ell^1 }} \\
			&\qquad \times \left( \frac{\norm{ f - f^s }_{ L^2 }}{\norm{ f }_{ L^2 }} + \norm{ a - a^s }_{ L^\infty } + \left( \frac{\norm{ \hat{ a } - \hat{ a }_\vec{ 0 } }_{ \ell^1 }}{a_\mathrm{min} - 2 \norm{ \hat{ a } - \hat{ a }_\vec{ 0 } }_{ \ell^1 } - \norm{ \hat{ a } - \hat{ \vec{ a } }^s }_{ \ell^1 }}  \right)^{ N + 1 } \right).
	\end{align*}
\end{cor}
\begin{proof}
	Since $ \hat{ a }' = \hat{ a }\restrict{ \supp \hat{ \vec{ a } }^s } $,
	\begin{gather*}
		\norm{ a - a' }_{ L^\infty } \leq \norm{ \hat{ a } - \hat{ a }' }_{ \ell^1 } \leq \norm{ \hat{ a } - \hat{ \vec{ a } }^s }_{ \ell^1 },\\
		\norm{ a' - a^s }_{ L^\infty } \leq \norm{ \hat{ a }' - \hat{ \vec{ a } }^s }_{ \ell^1 } \leq \norm{ \hat{ a } - \hat{ \vec{ a } }^s }_{ \ell^1 },
	\end{gather*}
	and analogously to show that $ \norm{ f - f' }_{ L^2 } $ and $ \norm{ f' - f^s }_{ L^2 } $ are bounded above by $ \norm{ f - f^s }_{ L^2 } $.
	Additionally, 
	\begin{gather*}
		a^s \geq a - \norm{ a - a^s }_{ L^\infty } \geq a - \norm{ \hat{ a } - \hat{ \vec{ a } }^s }_{ \ell^1 } \text{ and} \\
		a' \geq a - \norm{ a - a' }_{ L^\infty } \geq a - \norm{ \hat{ a } - \hat{ \vec{ a } }^s }_{ \ell^1 }
	\end{gather*}
	giving $ \min(a^s_\mathrm{min}, a'_\mathrm{min}) \geq a_\mathrm{min} - \norm{ \hat{ a } - \hat{ \vec{ a } }^s }_{ \ell^1 } $.
	The rest follows from applications of \eqref{eq:StabilityEstimate} and rearranging.
\end{proof}

\begin{rem}
	Though this final bound is difficult to parse, we can focus our attention on the final factor 
	\begin{equation}
		\label{eq:DrivingErrorFactor}
		\frac{\norm{ f - f^s }_{ L^2 }}{\norm{ f }_{ L^2 }} + \norm{ a - a^s }_{ L^\infty } + \left( \frac{\norm{ \hat{ a } - \hat{ a }_\vec{ 0 } }_{ \ell^1 }}{a_\mathrm{min} - 2 \norm{ \hat{ a } - \hat{ a }_\vec{ 0 } }_{ \ell^1 } - \norm{ \hat{ a } - \hat{ \vec{ a } }^s }_{ \ell^1 }}  \right)^{ N + 1 },
	\end{equation}
	since the other factors are more or less fixed.
	The first two terms are respectively controlled by having good SFT approximations to $  f $ in the $ L^2 $ norm and $ a $ in the  $ L^\infty $ norm.
	In our algorithm, these terms can be reduced by increasing the bandwidth $ K $ and the sparsity $ s $.
	As a reminder, the errors in these approximations given in Theorem~\ref{thm:SFTRecovery} are near optimal, as
	\begin{equation*}
			\norm{ f - f^s }_{ L^2 } \leq (25 + 3K) \sqrt{ s }\norm{ \hat{ f } - \left( \hat{ f }\restrict{ K } \right)_s^\mathrm{opt} }_{ \ell^1 } \text{ and }
			\norm{ a - a^s }_{ L^\infty } \leq (33 + 4K) \norm{ \hat{ a } - \left( \hat{ a }\restrict{ K } \right)_s^\mathrm{opt} }_{ \ell^1 }
	\end{equation*}
	with high probability.
	
	The final term is controlled by properties of $ a $ as well as the final stamping level used.
	Overall, the convergence is exponential in $ N $, the stamping level.
	This convergence is accelerated as the base of the exponent decreases: effectively, this happens as the diffusion coefficient approaches a large constant.
	Indeed, the numerator can be thought of as an upper bound for the absolute deviation of $ a $ from its mean while the denominator grows with the minimum of $ a $.
\end{rem}

\begin{rem}
	The computational complexity of Algorithm~\ref{alg:SparseSFT} is
	\begin{equation*}
		\mathcal{O}\left(ds \log^4(dK \max(K, s)) + \max(s, 2N + 1)^{ 3 \min(s, 2N + 1)}\right).
	\end{equation*}
	This is due to the two SFTs and a matrix solve of a $ \abs{ \mathcal{S}^N } \times \abs{ \mathcal{S}^N } $ system.
	Note that computing the stamping set can be done by enumerating the frequencies using the techniques in Lemma~\ref{lem:StampSetCardinality} and therefore is subject to the same upper bound as given in Lemma~\ref{lem:StampSizeUpperBound} for a stamp set's cardinality.
	Recall also that the SFT complexity can be tuned to produce SFT approximations satisfying the above bounds higher probability.

	We do not analyze the complexity of the matrix solve in depth, and instead resort to the upper bound given by Gaussian elimination on the dense matrix, $ \mathcal{O}\left( \max(s, 2N + 1)^{ 3 \min(s, 2N + 1) } \right) $.
	However, $ \vec{ L }_{ N, s } $ is relatively sparse for larger stamping levels.
	As the capabilities of sparse solvers depend strongly on analyzing the graph connecting interacting rows in $ \vec{ L }_{ N, s } $ (cf.\ \cite[Chapter~11]{golub_matrix_2013}), we expect that the analysis of an efficient sparse solver could be carried out using much of the same analysis of stamping sets performed in Section~\ref{sec:stamping_sets}.
\end{rem}

\begin{rem}
	\label{rem:ADR}
	This paper considers the theory for solving the simple diffusion equation \eqref{eq:DiffusionEquation}.
	However, these techniques extend to more complex advection-diffusion-reaction (ADR) equations.
	The test problem is then
	\begin{equation}
		\label{eq:ADR}
		- \nabla \cdot (a(\vec{ x }) \nabla u(\vec{ x })) + \vec{ b }(\vec{ x }) \cdot \nabla u(\vec{ x }) + c(\vec{ x }) u(\vec{ x }) = f(\vec{ x }) \text{ for all } \vec{ x } \in \T^3.
	\end{equation}
	As before $ a, f, u: \T^d \rightarrow \R $ are the diffusion coefficient, forcing function, and solution respectively.
	These are now joined by an advection field $ \vec{ b }: \T^d \rightarrow \R^d $ and an additional reaction coefficient $ c: \T^d \rightarrow \R $.
	For more on the properties and well-posedness of this periodic ADR equation, we refer to \cite{brugiapaglia_waveletfourier_2020}.

	Adapting Algorithm~\ref{alg:SparseSFT} for solving ADR equations requires two modifications:
	\begin{enumerate}
		\item When computing the approximations $ \hat{ \vec{ a } }^s, \hat{ \vec{ f } }^s $ via SFT, additionally compute $ \hat{ \vec{ b } }^s := (\hat{ \vec{ b } }_j^s)_{ j = 1 }^d $, an approximation to the Fourier coefficients of each component of $ \vec{ b } $, and compute $ \hat{ \vec{ c } }^s $, an approximation to $ \hat{ c } $.
		\item Redefine the ``stamp'' used to define $ \mathcal{S}^N[\hat{ \vec{ a } }^s](\supp(\hat{ \vec{ f } }^s)) $ by including the supports of $ \hat{ \vec{ b } }^s $ and $ \hat{ \vec{ c } }^s $.
			Mathematically, we define
			\begin{equation*}
				\mathcal{S}^N[\hat{ \vec{ a } }^s, \hat{ \vec{ b } }^s, \hat{ \vec{ c } }^s](\supp(\hat{ \vec{ f } }^s)) := 
				\begin{cases}
					\supp(\hat{ \vec{ f } }^s) & \text{if } N = 0 \\
					\mathcal{S}^{ N - 1 } + \supp(\hat{ \vec{ a } }^s) + \sum_{ j = 1 }^d \supp(\hat{ \vec{ b } }_j^s) + \supp(\hat{ \vec{ c } }^s) & \text{if } N > 0
				\end{cases}
			\end{equation*}
			where, as usual, we suppress the Fourier coefficients when clear from context.
	\end{enumerate}

	The convergence analysis for this method is much the same as that leading to Corollary~\ref{cor:SpectralConvergenceWithSFT} where terms like $ \norm{ a - a^s }_{ L^\infty } $ are replaced by $ \max\left\{ \norm{ a - a^s }_{ L^\infty }, \norm{ \norm{ \vec{ b } - \vec{ b }^s }_{ \ell^2 } }_{ L^\infty }, \norm{ c - c^s }_{ L^\infty } \right\} $ and similarly for the mean-zero version of $ a $ used in the exponentially decaying term.
	For full details see \cite{gross_dissertation_2023}.
\end{rem}

\section{Numerics}\label{sec:numerics}

This section gives examples of the algorithm summarized above applied to various problems.
We begin with an overview of our implementation as well as some techniques used to evaluate the accuracy of our approximations.
We then present solutions to univariate and very high-dimensional multiscale problems with both exactly sparse and Fourier-compressible data.
We then close with an extension of our methods to a three-dimensional advection-diffusion-reaction equation.

\subsection{Code and testing overview}\label{sub:code_and_testing_overview}

We implement Algorithm~\ref{alg:SparseSFT} described above in MATLAB using an object-oriented approach, with all code publicly available.\footnote{\url{https://gitlab.com/grosscra/SparseADR}}
All SFTs are computed using the rank-1 lattice sparse Fourier code from \cite{gross_sparse_2021}.\footnote{this code is publicly available at \url{https://gitlab.com/grosscra/Rank1LatticeSparseFourier}}

In order to evaluate the quality of our approximations, we need to choose an appropriate metric.
Letting $ u^{ s, N } $ be the approximation returned by our algorithm, the ideal choice would be $ \norm{ u - u^{ s, N } }_{ H } $.
However, for the types of problems we will be investigating, the true solution $ u $ is unavailable to us.
Instead, we will use a proxy that takes advantage of the stability result in Proposition~\ref{prop:ExistenceUniquenessStability}.

\begin{lem}
	\label{lem:ProxyError}
	Let $ u $ be the true solution to \eqref{eq:GalerkinForm} and $ u^{ s, N } $ be the approximation returned by solving \eqref{eq:SparseSpectralSFTEquation}.
	Define $ \hat{ f }^{ s, N } := L[\hat{ a }]\hat{ u }^{ s, N } $ with $ f^{ s, N } = \mathcal{L}[a]u^{ s, N } $.
	Then
	\begin{equation*}
		\norm{ u - u^{ s, N } }_{ H } \leq \frac{\norm{ f - f^{ s, N } }_{ L^2 }}{a_\mathrm{min}} = \frac{\norm{ \hat{ f } - \hat{ f }^{ s, N } }_{ \ell^2 }}{a_\mathrm{min}}.
	\end{equation*}
\end{lem}
\begin{proof}
	The result follows from the fact that $ \hat{ u } - \hat{ u }^{ s, N } $ solves $ L[\hat{ a }]\left(\hat{ u } - \hat{ u }^{ s, N }\right) = \hat{ f } - L[\hat{ a }]\hat{ u }^{ s, N } = \hat{ f } - \hat{ f }^{ s, N }$ and applying Proposition~\ref{prop:ExistenceUniquenessStability}.
\end{proof}

In the sequel, we will ignore $ a_\mathrm{min} $ since we are mostly interested in convergence properties in $ s $ and $ N $ and we will compute the relative error
\begin{equation*}
	\frac{\norm{ f - f^{ s, N } }_{ L^2 }}{\norm{ f }_{ L^2 }} \text{ or } \frac{\norm{ \hat{ f } - \hat{ f }^{ s, N } }_{ \ell^2 }}{\norm{ \hat{ f } }_{ \ell^2 }}
\end{equation*}
as our proxy instead.
Whenever $ \hat{ f } $ and $ \hat{ a } $ are exactly sparse, the numerator of the second term can be computed exactly due to the fact that $ \supp(\hat{ f }^{ s, N }) $ is known to be contained in $ \mathcal{S}^{ N + 1 } $ (cf.\ Proposition~\ref{prop:LExpansion}).
However, in the non-sparse setting, even though $ f - f^{ s, N } $ can be evaluated pointwise, computing an accurate approximation of its norm on $ \T^d $ is challenging for large $ d $.
For this reason, we approximate the norm via Monte Carlo sampling.
We also furnish the cases where exactly computing $ \norm{ \hat{ f } - \hat{ f }^{ s, N } }_{ \ell^2 } $ is possible with the pointwise Monte Carlo estimates to show that in practice, Monte Carlo sampling does as well as the exact computation.

\subsection{Univariate compressible}\label{sub:univariate_compressible}

We begin by replicating the lone numerical example of solving an elliptic problem in \cite[Section~5.1]{daubechies_sparse_2007}.
In this case, we solve the univariate problem
\begin{equation}
	\begin{gathered}
		\label{eq:ODE}
		- (a(x) u'(x))' = f(x) \text{ for all } x \in \T, \text{ where }\\ 
		a(x) = \frac{1}{10} \exp \left( \frac{0.6 + 0.2 \cos(2 \pi x)}{1 + 0.7 \sin(256 \pi x)} \right), \quad f(x) = \exp(-\cos(2\pi x)) - \int_\T \exp(-\cos(2\pi x)) \d x
	\end{gathered}
\end{equation}
(note that the only difference from \cite{daubechies_sparse_2007} is that we use the domain $ \T = [0, 1] $ rather than $ [0, 2\pi] $).
This data is not Fourier sparse, but is compressible.
In the original paper, a bandwidth of $ K = 1\,536 $ is considered and approximations with $ 9 $ and $ 17 $ Fourier coefficients are used.

We first construct a high accuracy approximation of the solution to \eqref{eq:ODE} by numerically integrating on an extremely fine mesh of $ 10\,000 $ points.
This allows us to forgo our proxy error described in Lemma~\ref{lem:ProxyError}.
As in \cite{daubechies_sparse_2007}, the bandwidth of our SFT used is set to $ K = 1\,536 $.
Due to our SFT returning a $ 2s $ sparse approximation, we use $ s = 4 $ and $ s = 8 $ to compare with the $ 9 $ and $ 17 $ terms respectively considered in the original paper, and also provide an example with $ s = 12 $.
We set the stamping level to $ N = 1 $ throughout, which, as discussed in the introduction, is similar to the technique used in \cite{daubechies_sparse_2007}.

\begin{figure}[ht]
	\centering
	\begin{tikzpicture}[scale=1]
		\datavisualization [
			scientific axes={width=0.5\linewidth},
			x axis={attribute = s,
				label={$ s $ (sparsity)},
				ticks={about=3, step=4, minor steps between steps=3}},
			y axis={attribute = y,
				label={Relative error},
				logarithmic,
				ticks={minor steps between steps=9,
					tick typesetter/.code=\nomantissatypeset{##1}}},
			legend/.list={south west inside, transparent},
			legend entry options/default label in legend path/.style=
				straight label in legend line,
			visualize as line/.list={gridErrors,gridErrorsDeriv,gridErrorsProxy},
			gridErrors={label in legend={text=$L^2$}},
			gridErrorsDeriv={label in legend={text=$H^1$}},
			gridErrorsProxy={label in legend={text=Proxy error}},
			style sheet/.list={vary dashing, strong colors, vary mark},
		]
		data [set=gridErrors, headline={s, N, y}, read from file=Results/Daubechies1DSparsity/gridErrors_s_N_y.csv]
		data [set=gridErrorsDeriv, headline={s, N, y}, read from file=Results/Daubechies1DSparsity/gridErrorsDeriv_s_N_y.csv]
		data [set=gridErrorsProxy, headline={s, N, y}, read from file=Results/Daubechies1DSparsity/gridErrorsProxy_s_N_y.csv];
\end{tikzpicture}
\caption{Errors in approximating the solution to \eqref{eq:ODE}.}\label{fig:Daubechies1DErrors}
\end{figure}
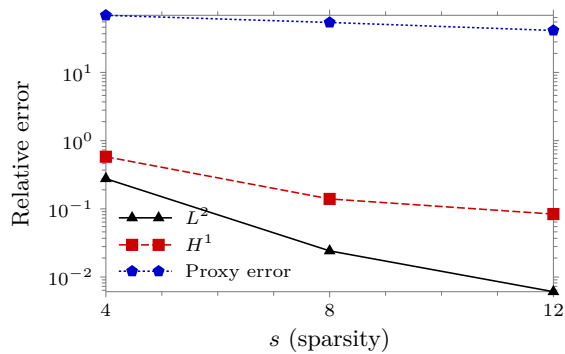

\begin{figure}[ht]
	\begin{subfigure}{0.49\textwidth}
	\begin{center}
	\begin{tikzpicture}[scale=1]
		\datavisualization[
			scientific axes={width=0.8\textwidth},
			x axis={attribute = x, ticks=some},
			y axis={attribute = y, ticks={some, tick typesetter/.code=\fixedzerofilltypeset{2}{##1}}},
			legend=south west inside,
			legend entry options/default label in legend path/.style=
				straight label in legend line,
			visualize as smooth line/.list={trueSol, s4, s8, s12},
			trueSol={label in legend={text=$u$}},
			s4 ={label in legend={text=$u^{ 4, 1 }$}},
			s8 ={label in legend={text=$u^{ 8, 1 }$}},
			s12 ={label in legend={text=$u^{ 12, 1 }$}},
			style sheet/.list = {strong colors, vary dashing}
		]
		data [set=trueSol, headline={x, y}, read from file=Results/Daubechies1DSparsity/trueSol_x_y.csv]
		data [set=s4, headline={x, y}, read from file=Results/Daubechies1DSparsity/recoveredSol_s_4_N_1536_x_y.csv]
		data [set=s8, headline={x, y}, read from file=Results/Daubechies1DSparsity/recoveredSol_s_8_N_1536_x_y.csv]
		data [set=s12, headline={x, y}, read from file=Results/Daubechies1DSparsity/recoveredSol_s_12_N_1536_x_y.csv];
	\end{tikzpicture}
	\end{center}
\caption{Approximate solutions of \eqref{eq:ODE}.}
	\label{fig:Daubechies1DSolution}
	\end{subfigure}
	\hfill
	\begin{subfigure}{0.49\textwidth}
	\begin{center}
	\begin{tikzpicture}[scale=1]
		\datavisualization[
			scientific axes={width=0.8\textwidth},
			x axis={attribute = x, ticks={few, tick typesetter/.code=\fixedzerofilltypeset{3}{##1}}},
			y axis={attribute = y, ticks={some, tick typesetter/.code=\fixedzerofilltypeset{2}{##1}}},
			legend={above of={x=0.682,y=.1}, transparent},
			legend entry options/default label in legend path/.style=
				straight label in legend line,
			visualize as smooth line/.list={trueDeriv, s4, s8, s12},
			trueDeriv={label in legend={text=$u'$}},
			s4 ={label in legend={text=$(u^{ 4, 1 })'$}},
			s8 ={label in legend={text=$(u^{ 8, 1 })'$}},
			s12 ={label in legend={text=$(u^{ 12, 1 })'$}},
			style sheet/.list = {strong colors, vary dashing}
		]
		data [set=trueDeriv, headline={x, y}, read from file=Results/Daubechies1DSparsity/trueDeriv_x_y.csv]
		data [set=s4, headline={x, y}, read from file=Results/Daubechies1DSparsity/recoveredDeriv_s_4_N_1536_x_y.csv]
		data [set=s8, headline={x, y}, read from file=Results/Daubechies1DSparsity/recoveredDeriv_s_8_N_1536_x_y.csv]
		data [set=s12, headline={x, y}, read from file=Results/Daubechies1DSparsity/recoveredDeriv_s_12_N_1536_x_y.csv];
	\end{tikzpicture}
	\end{center}
\caption{Detail of approximate derivatives of \eqref{eq:ODE}.}
	\label{fig:Daubechies1DDerivative}
	\end{subfigure}
	\caption{Qualitative results.}
	\label{fig:Daubechies1DErrorsQualitative}
\end{figure}
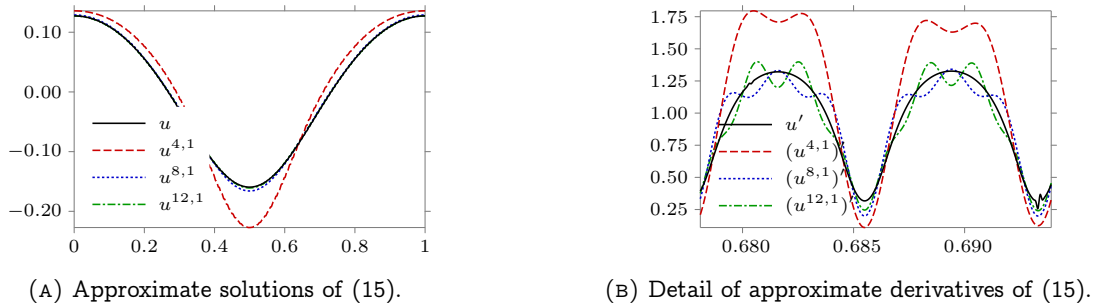

The relative errors approximated in $ L^2 $ and $ H^1 $ are given in Figure~\ref{fig:Daubechies1DErrors}.
The original paper does not give numerical results, and instead, gives qualitative results, comparing the approximate solutions and their derivatives with the true solution and its derivative.
We have replicated this qualitative analysis in Figure~\ref{fig:Daubechies1DErrorsQualitative} with similar results.

Figure~\ref{fig:Daubechies1DErrors} also shows the error computed via the proxy described by Lemma~\ref{lem:ProxyError}, and in particular, how pessimistic the proxy error can be. 
In this case, the small errors in the derivative (visualized in Figure~\ref{fig:Daubechies1DDerivative}) are compounded by passing the approximate solution through the operator where $ a' $ is often large relative to $ a $.
In future examples, we will see that the convergence of the proxy error is much more tolerable.

\subsection{Multivariate exactly sparse}\label{sub:multivariate_exactly_sparse}

\subsubsection{Low sparsity}\label{ssub:low_sparsity}

Moving to the multivariate case, we start with a simple example with exactly sparse data.
Our goal is to solve 
\begin{equation}
	\begin{gathered}
		\label{eq:PDELowSparsity}
		- \nabla \cdot (a(\vec{ x }) \nabla u(\vec{ x })) = f(\vec{ x }) \text{ for all } \vec{ x } \in \T^d, \text{ where }\\ 
		a(\vec{ x }) = \hat{ a }_0 + c_{ a } \cos(2 \pi \vec{ k }_a \cdot \vec{ x }), \quad f(x) = \sin(2 \pi \vec{ k }_f \cdot \vec{ x }).
	\end{gathered}
\end{equation}
We draw $ c_a \sim \Unif \left(  [-1, 1] \right) $, keep it constant for each dimension, and set $ \hat{ a }_\vec{ 0 } = 4 $ so that our problem remains elliptic (in the specific example below, $ c_a \approx -0.6 $). 
For dimensions varying from $ d = 1 $ to $ d = 1\,024 $, we then draw $ \vec{ k }_a, \vec{ k }_f \sim \Unif \left( [-499, 500]^d \cap \Z^d \right) $.
The PDE \eqref{eq:PDELowSparsity} is then solved for stamping levels $ N = 1,\ldots, 5 $.
The bandwidth of the SFT is set to $ 1000 $ and the sparsity is set to $ 2 $. 
We then compute a Monte Carlo approximation of the proxy error choosing $ 200 $ points drawn uniformly from $ \T^d $ and also compute the proxy error exactly by virtue of the sparsity of $ a $ and $ f $.
The results are given in Figure~\ref{fig:oneTermDiffusion}.

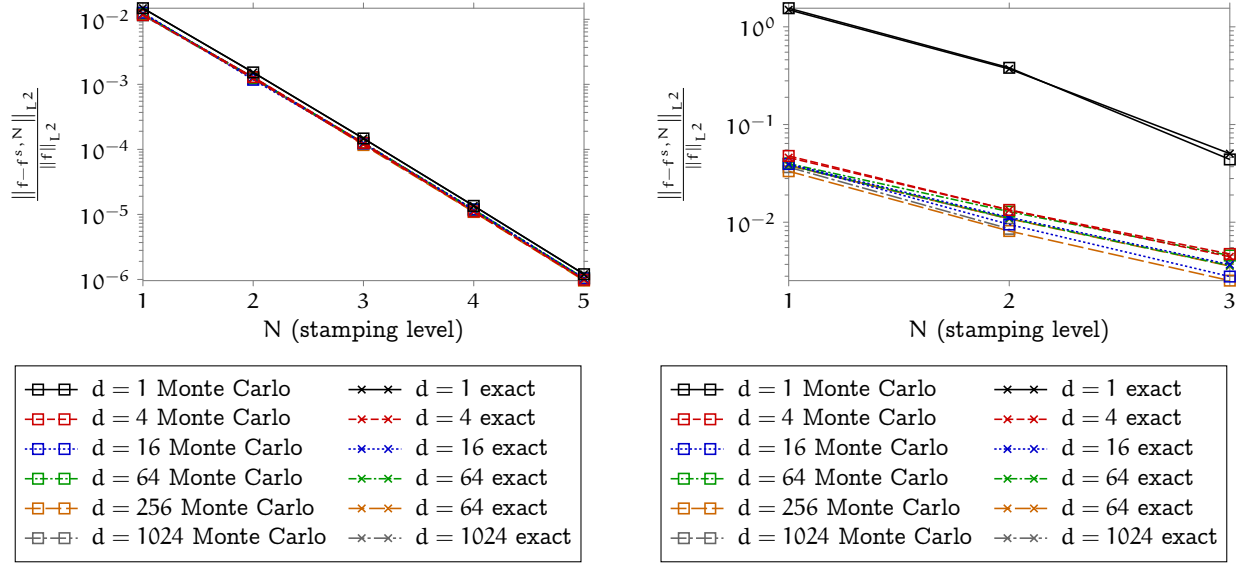
\begin{figure}[ht]
	\begin{subfigure}[t]{0.48\linewidth}
		\centering
		\begin{tikzpicture}[scale=1]
			\datavisualization [
				scientific axes={width=0.74\linewidth},
				x axis={attribute = n,
					label={$ N $ (stamping level)},
					ticks={few}},
				y axis={attribute = y,
					label={$ \frac{\norm{f - f^{ s, N } }_{ L^2 }}{\norm{ f }_{ L^2 }}$ },
					logarithmic,
					ticks={minor steps between steps=9,
						tick typesetter/.code=\nomantissatypeset{##1}}},
				legend={below, down then right, columns = 2, matrix node style={draw}},
				legend entry options/default label in legend path/.style=
					straight label in legend line,
				visualize as line/.list={
					gridErrors1,
					fourierErrors1,
					gridErrors4,
					fourierErrors4,
					gridErrors16,
					fourierErrors16,
					gridErrors64,
					fourierErrors64,
					gridErrors256,
					fourierErrors256,
					gridErrors1024,
					fourierErrors1024
				},
				gridErrors1={label in legend={text={$d = 1$ Monte Carlo}}},
				gridErrors4={label in legend={text={$d = 4$ Monte Carlo}}},
				gridErrors16={label in legend={text={$d = 16$ Monte Carlo}}},
				gridErrors64={label in legend={text={$d = 64$ Monte Carlo}}},
				gridErrors256={label in legend={text={$d = 256$ Monte Carlo}}},
				gridErrors1024={label in legend={text={$d = 1024$ Monte Carlo}}},
				fourierErrors1={label in legend={text={$d = 1$ exact}}},
				fourierErrors4={label in legend={text={$d = 4$ exact}}},
				fourierErrors16={label in legend={text={$d = 16$ exact}}},
				fourierErrors64={label in legend={text={$d = 64$ exact}}},
				fourierErrors256={label in legend={text={$d = 64$ exact}}},
				fourierErrors1024={label in legend={text={$d = 1024$ exact}}},
				style sheet/.list={alternate dashing, color pair, alternate mark},
			]
			data [set=gridErrors1, headline={n, y}, read from file=Results/oneTermDiffusion/gridErrors_d1_n_y.csv]
			data [set=gridErrors4, headline={n, y}, read from file=Results/oneTermDiffusion/gridErrors_d4_n_y.csv]
			data [set=gridErrors16, headline={n, y}, read from file=Results/oneTermDiffusion/gridErrors_d16_n_y.csv]
			data [set=gridErrors64, headline={n, y}, read from file=Results/oneTermDiffusion/gridErrors_d64_n_y.csv]
			data [set=gridErrors256, headline={n, y}, read from file=Results/oneTermDiffusion/gridErrors_d256_n_y.csv]
			data [set=gridErrors1024, headline={n, y}, read from file=Results/oneTermDiffusion/gridErrors_d1024_n_y.csv]
			data [set=fourierErrors1, headline={n, y}, read from file=Results/oneTermDiffusion/fourierErrors_d1_n_y.csv]
			data [set=fourierErrors4, headline={n, y}, read from file=Results/oneTermDiffusion/fourierErrors_d4_n_y.csv]
			data [set=fourierErrors16, headline={n, y}, read from file=Results/oneTermDiffusion/fourierErrors_d16_n_y.csv]
			data [set=fourierErrors64, headline={n, y}, read from file=Results/oneTermDiffusion/fourierErrors_d64_n_y.csv]
			data [set=fourierErrors256, headline={n, y}, read from file=Results/oneTermDiffusion/fourierErrors_d256_n_y.csv]
			data [set=fourierErrors1024, headline={n, y}, read from file=Results/oneTermDiffusion/fourierErrors_d1024_n_y.csv];
		\end{tikzpicture}
\caption{Proxy error solving \eqref{eq:PDELowSparsity} with $ d = 1$, $ 4 $, $ 16 $, $ 64 $, $ 256 $, $ 1\,024 $ and $ N = 1, \ldots, 5 $.}\label{fig:oneTermDiffusion}
	\end{subfigure}
	\hfill
	\begin{subfigure}[t]{0.48\linewidth}
		\centering
		\begin{tikzpicture}[scale=1]
			\datavisualization [
				scientific axes={width=0.74\linewidth},
				x axis={attribute = n,
					label={$ N $ (stamping level)},
					ticks={step=1}},
				y axis={attribute = y,
					label={$ \frac{\norm{f - f^{ s, N } }_{ L^2 }}{\norm{ f }_{ L^2 }}$ },
					logarithmic,
					ticks={minor steps between steps=9,
						tick typesetter/.code=\nomantissatypeset{##1}}},
				legend={below, down then right, columns = 2, matrix node style={draw}},
				legend entry options/default label in legend path/.style=
					straight label in legend line,
				visualize as line/.list={
					gridErrors1,
					fourierErrors1,
					gridErrors4,
					fourierErrors4,
					gridErrors16,
					fourierErrors16,
					gridErrors64,
					fourierErrors64,
					gridErrors256,
					fourierErrors256,
					gridErrors1024,
					fourierErrors1024
				},
				gridErrors1={label in legend={text={$d = 1$ Monte Carlo}}},
				gridErrors4={label in legend={text={$d = 4$ Monte Carlo}}},
				gridErrors16={label in legend={text={$d = 16$ Monte Carlo}}},
				gridErrors64={label in legend={text={$d = 64$ Monte Carlo}}},
				gridErrors256={label in legend={text={$d = 256$ Monte Carlo}}},
				gridErrors1024={label in legend={text={$d = 1024$ Monte Carlo}}},
				fourierErrors1={label in legend={text={$d = 1$ exact}}},
				fourierErrors4={label in legend={text={$d = 4$ exact}}},
				fourierErrors16={label in legend={text={$d = 16$ exact}}},
				fourierErrors64={label in legend={text={$d = 64$ exact}}},
				fourierErrors256={label in legend={text={$d = 64$ exact}}},
				fourierErrors1024={label in legend={text={$d = 1024$ exact}}},
				style sheet/.list={alternate dashing, color pair, alternate mark},
			]
			data [set=gridErrors1, headline={n, y}, read from file=Results/highSparsityDiffusion/gridErrors_d1_n_y.csv]
			data [set=gridErrors4, headline={n, y}, read from file=Results/highSparsityDiffusion/gridErrors_d4_n_y.csv]
			data [set=gridErrors16, headline={n, y}, read from file=Results/highSparsityDiffusion/gridErrors_d16_n_y.csv]
			data [set=gridErrors64, headline={n, y}, read from file=Results/highSparsityDiffusion/gridErrors_d64_n_y.csv]
			data [set=gridErrors256, headline={n, y}, read from file=Results/highSparsityDiffusion/gridErrors_d256_n_y.csv]
			data [set=gridErrors1024, headline={n, y}, read from file=Results/highSparsityDiffusion/gridErrors_d1024_n_y.csv]
			data [set=fourierErrors1, headline={n, y}, read from file=Results/highSparsityDiffusion/fourierErrors_d1_n_y.csv]
			data [set=fourierErrors4, headline={n, y}, read from file=Results/highSparsityDiffusion/fourierErrors_d4_n_y.csv]
			data [set=fourierErrors16, headline={n, y}, read from file=Results/highSparsityDiffusion/fourierErrors_d16_n_y.csv]
			data [set=fourierErrors64, headline={n, y}, read from file=Results/highSparsityDiffusion/fourierErrors_d64_n_y.csv]
			data [set=fourierErrors256, headline={n, y}, read from file=Results/highSparsityDiffusion/fourierErrors_d256_n_y.csv]
			data [set=fourierErrors1024, headline={n, y}, read from file=Results/highSparsityDiffusion/fourierErrors_d1024_n_y.csv];
		\end{tikzpicture}
\caption{Proxy error solving \eqref{eq:PDELowSparsity} with diffusion coefficient \eqref{eq:PDEHighSparsityDiffusion} in dimensions $ d = 4$, $ 64 $, $ 1\,024 $ and stamping levels $ N = 1, \ldots, 3 $.}\label{fig:diffusionHighSparsity}
	\end{subfigure}
	\caption{Solving diffusion equation with exactly sparse data.}
	\label{fig:sparseDiffusion}
\end{figure}

We see that the results do not depend on the dimension of the problem.
Since all dependence on $ d $ is in the runtime of the SFT, we also observe that in practice, after the SFTs of the data have been computed, re-solving the problem on different stamping levels takes about the same amount of time for each $ d $.
The error also converges exponentially in the stamping level as suggested by the theoretical error guarantees.
Notably, we also see that the Monte Carlo approximation with $ 200 $ points captures the same proxy error as the exact computation.

\subsubsection{High sparsity}\label{ssub:high_sparsity}

We expand on the exactly sparse case by testing a diffusion coefficient with much higher sparsity.
Here, we solve \eqref{eq:PDELowSparsity} with 
\begin{equation}
	\label{eq:PDEHighSparsityDiffusion}
	a(\vec{ x }) = \hat{ a }_\vec{ 0 } + \sum_{ \vec{ k } \in \mathcal{I}_a } c_\vec{ k }\cos(2 \pi \vec{ k } \cdot \vec{ x }).
\end{equation}
The vector of coefficients is drawn as $ \vec{ c } \sim \Unif\left( [-1, 1]^{ 25 } \right) $ once and reused in each test.
For every $ d $, the frequencies $ \vec{ k } \in \mathcal{I}_a $ are each drawn uniformly from $ [-499, 500]^d \cap \Z^d $ as before with $ |\mathcal{I}_a| = 25 $.
Here $ \hat{ a }_0 = 4 \left\lceil \norm{ \vec{ c } }_2 \right\rceil $ to ensure ellipticity.
Again, the bandwidth of the SFT algorithm is set to $ 1\,000 $, but the sparsity is now fixed to $ 26 $.
The results are given in Figure~\ref{fig:diffusionHighSparsity}

Again, we see that the results do not depend on the spatial dimension except for the notable example of $ d = 1 $.
The $ d = 1 $ case suffers from similar issues in a pessimistic proxy error as in Figure~\ref{fig:Daubechies1DErrors}.
Specifically, the right hand-side for this example was generated with frequency $ k_f = -10 $ and is therefore relatively low-frequency.
Thus, the high-frequency modes leading to errors in the approximate solution are amplified by the high-frequencies in $ a $ when computing $ f^{ s, N } $.
Indeed, in further experiments (not pictured here), increasing the frequencies of $ f $ or decreasing the frequencies of $ a $ result in a lower proxy error.

For the other dimensions, the slight offsets in the exact proxy error can be attributed to the randomized frequencies as well as slight variations in the randomized SFT code.
We do see slightly more variance in the proxy error computed using Monte Carlo sampling however.
This is to be expected for data with more varied frequency content, and as such, in future experiments, we increase the number of sampling points.

Note that because we consider sparsity much larger than the stamping level, the computational and memory complexity of the stamping and solution step is much higher.
As suggested by Lemma~\ref{lem:StampSizeUpperBound}, the size of the resulting stamp set (and therefore the necessary matrix solve) in the largest case is at most $ 7 \cdot 52^{ 7 } \approx 7 \times 10^{ 12 } $ which pushes the memory boundaries of our computational resources.

\subsection{Multivariate compressible}\label{sub:multivariate_compressible}

In order to test Fourier-compressible data which is not exactly sparse, we use a series of tensorized, periodized Gaussians.
Here, we present the only details necessary to demonstrate our algorithm's effectiveness on Fourier-compressible data, but for a fuller treatment on the Fourier properties of periodized Gaussians, see e.g., \cite[Section~2.1]{merhi_new_2019}.

Here, we define the periodic Gaussian $ G_{ r } : \T \rightarrow \R $ by
\begin{equation*}
	G_{ r }(x) = \frac{\sqrt{ 2\pi }}{r}\sum_{ m = -\infty }^\infty \e^{ - \frac{(2\pi)^2(x - m)}{2 r^2} }
\end{equation*}
where the dilation-type parameter $ r $ allows us to control the effective support of $ \hat{ G }_r $.
In practice, we truncate the infinite sum to $ m \in \{-10, \ldots, 10\} $ as additional terms do not change the output up to machine precision.
Note here that the nonstandard multiplicative factors help control the behavior of the function in frequency rather than space.
Given a multivariate modulating frequency $ \vec{ k } \in \Z^d $, we define the modulated, tensorized, periodic Gaussian by
\begin{equation*}
	G_{ r, \vec{ k } }(\vec{ x }) = \prod_{ j = 1 }^d \e^{ 2\pi \i k_i x_i } G_r(x_i).
\end{equation*}
Finally, given a set of frequencies $ \mathcal{I} \subset \Z^d $, dilation parameters $ \vec{ r } \in \R_+^\mathcal{I} $, and coefficients $ \vec{ c } \in \R^{ \mathcal{I} } $, we can define Gaussian series
\begin{equation*}
	G_{ \vec{ c }, \vec{ r } }^{ \mathcal{I} }(\vec{ x }) := \sum_{ \vec{ k } \in \mathcal{I} } c_\vec{ k } G_{ r_\vec{ k }, \vec{ k } }(\vec{ x }).
\end{equation*}

Depending on the severity of the dilations chosen (i.e., $ r_\vec{ k } \gg 1 $), this can well approximate a Fourier series with frequencies in $ \mathcal{I} $.
On the other hand, a less severe dilation results in Fourier coefficients with magnitudes forming less concentrated Gaussians centered around the ``frequencies'' $ \vec{ k } \in \mathcal{I} $ and $ -\vec{ k } $.
An example of a series with its associated Fourier transform is given in Figure~\ref{fig:gaussianExample}.

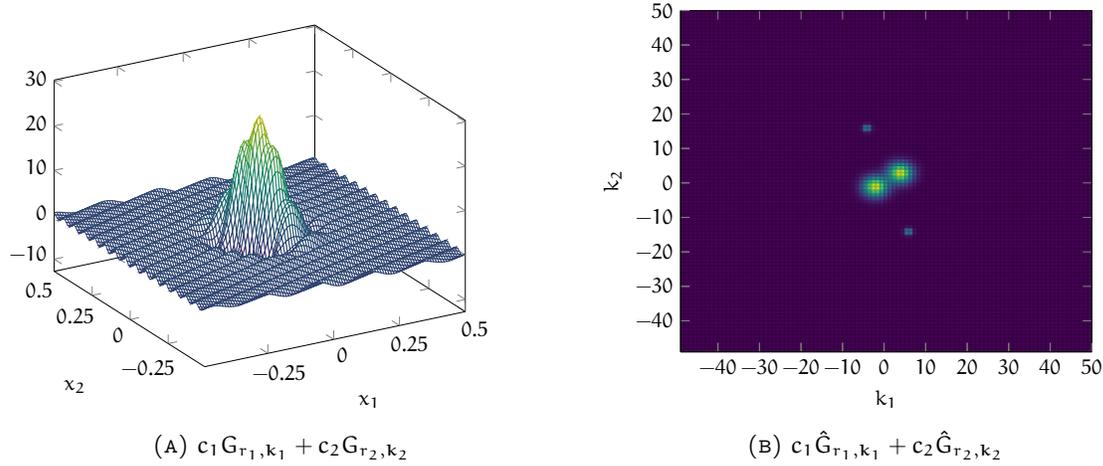
\begin{figure}[ht]
	\centering
	\begin{subfigure}[t]{0.45\linewidth}
		\begin{tikzpicture}[scale=0.8]
			\begin{axis}[
				view={-30}{30},
				xlabel=$x_1$,
				xtick distance=0.25,
				ylabel=$x_2$,
				ytick distance=0.25,
				ztick distance=10,
			]
				\addplot3 [
					surf, 
					fill=white,
					mesh/rows=100,
					colormap name={viridis},
				] table [col sep=comma] {Figures/gaussian_xN_100_yN_100.csv};
			\end{axis}	
		\end{tikzpicture}
		\caption{$c_1 G_{r_1, \vec{ k }_1} + c_2 G_{r_2, \vec{ k }_2}$}
		\label{fig:gaussianExampleFunction}
	\end{subfigure}
	\quad
	\begin{subfigure}[t]{0.45\linewidth}
		\begin{tikzpicture}[scale=0.8]
			\begin{axis}[
				view={0}{90},
				xlabel=$k_1$,
				xtick distance=10,
				ylabel=$k_2$,
				ytick distance=10,
			]
				\addplot3 [
					surf, 
					mesh/rows=100,
					colormap name={viridis},
				] table [col sep=comma] {Figures/gaussian_FFT_xN_100_yN_100.csv};
			\end{axis}	
		\end{tikzpicture}
		\caption{$c_1 \hat{ G }_{r_1, \vec{ k }_1} + c_2 \hat{ G }_{r_2, \vec{ k }_2}$}
		\label{fig:gaussianExampleFourierTranform}
	\end{subfigure}
\caption{An example Gaussian series with $ c_1 = c_2 = 1 $, $ r_1 = 0.5 $, $ r_2 = 2 $, $ \vec{ k }_1 = (3, 2) $, and $ \vec{ k }_2 = (-5, 15) $.
	The first term corresponds to the wider Gaussian shape and more spread out portions of the Fourier transform.
	The second term contributes to the highly oscillatory parts and the isolated spikes in the Fourier transform.}\label{fig:gaussianExample}
\end{figure}

In our first experiment, we fix $ d = 2 $ and vary both stamp level and sparsity to again solve \eqref{eq:PDELowSparsity}.
The diffusion coefficient in \eqref{eq:PDELowSparsity} is replaced with a two-term Gaussian series $ a = c_0 + G^{ \mathcal{I} }_{ \vec{ c }, \vec{ r } } $, where
\begin{equation*}
	\mathcal{I} \sim \Unif \left( \left( [-24, 25]^2 \cap \Z^2 \right)^2  \right), \quad \vec{ c } \sim \Unif\left([-1, 1]^2\right), \quad \vec{ r } = 1.1^2 \vec{ 1 }, \quad c_0 = 10 \left\lceil \norm{ \vec{ c } }_2 \right\rceil.
\end{equation*}
Note the increased constant factor from our previous examples to decrease the likelihood of sparse approximations of $ a $ not satisfying the ellipticity property.
The Fourier transform of the resulting $ a $ used for the following test is depicted in Figure~\ref{fig:gaussianSeriesDiffusionSparsityFFT} below.
The diffusion equation is then solved across various sparsities with increasing stamping level.
The bandwidth parameter of the SFT is set to $ K = 100 $ to account for the wider effective support of $ \hat{ a } $.
The Monte Carlo proxy error is computed with $ 1\,000 $ samples and depicted in Figure~\ref{fig:gaussianSeriesDiffusionSparsity}.

\begin{figure}[ht]
	\begin{subfigure}[t]{0.48\linewidth}
		\centering
		\begin{tikzpicture}[scale=0.8]
			\begin{axis}[
				view={0}{90},
				xlabel=$k_1$,
				xtick distance=10,
				ylabel=$k_2$,
				ytick distance=10,
			]
				\addplot3 [
					surf, 
					mesh/rows=100,
					colormap name={viridis},
				] table [col sep=comma] {Results/gaussianSeriesDiffusionSparsity/gaussian_FFT_xN_100_yN_100_.csv};
			\end{axis}	
		\end{tikzpicture}
\caption{The specific $\hat{ a }$ used in examples depicted in Figure~\ref{fig:gaussianSeriesDiffusionSparsity}.}\label{fig:gaussianSeriesDiffusionSparsityFFT}
	\end{subfigure}
	\hfill
	\begin{subfigure}[t]{0.48\linewidth}
		\centering
		\begin{tikzpicture}[scale=1]
			\datavisualization [
				scientific axes={width=0.74\linewidth},
				x axis={attribute = n,
					label={$ N $ (stamping level)},
					ticks={step=1}},
				y axis={attribute = y,
					label={$ \frac{\norm{f - f^{ s, N } }_{ L^2 }}{\norm{ f }_{ L^2 }}$ },
					logarithmic,
					ticks={minor steps between steps=9,
						tick typesetter/.code=\nomantissatypeset{##1}}},
				legend={below, down then right, columns = 2, matrix node style={draw}},
				legend entry options/default label in legend path/.style=
					straight label in legend line,
				visualize as line/.list={
					gridErrors2,
					gridErrors4,
					gridErrors8,
					gridErrors16,
					gridErrors32,
					gridErrors64
				},
				gridErrors2={label in legend={text={$s = 2$ Monte Carlo}}},
				gridErrors4={label in legend={text={$s = 4$ Monte Carlo}}},
				gridErrors8={label in legend={text={$s = 8$ Monte Carlo}}},
				gridErrors16={label in legend={text={$s = 16$ Monte Carlo}}},
				gridErrors32={label in legend={text={$s = 32$ Monte Carlo}}},
				gridErrors64={label in legend={text={$s = 64$ Monte Carlo}}},
				style sheet/.list={vary dashing, strong colors, vary mark},
				]
			data [set=gridErrors2, headline={n, y}, read from file=Results/gaussianSeriesDiffusionSparsity/gridErrors_s2_n_y.csv]
			data [set=gridErrors4, headline={n, y}, read from file=Results/gaussianSeriesDiffusionSparsity/gridErrors_s4_n_y.csv]
			data [set=gridErrors8, headline={n, y}, read from file=Results/gaussianSeriesDiffusionSparsity/gridErrors_s8_n_y.csv]
			data [set=gridErrors16, headline={n, y}, read from file=Results/gaussianSeriesDiffusionSparsity/gridErrors_s16_n_y.csv]
			data [set=gridErrors32, headline={n, y}, read from file=Results/gaussianSeriesDiffusionSparsity/gridErrors_s32_n_y.csv]
			data [set=gridErrors64, headline={n, y}, read from file=Results/gaussianSeriesDiffusionSparsity/gridErrors_s64_n_y.csv];
		\end{tikzpicture}
\caption{Proxy error solving \eqref{eq:PDELowSparsity} with Gaussian series diffusion coefficient with sparsity levels $ s = 2, 4, 8, 16, 32, 64 $, and stamping levels $ N = 1, \ldots, 3 $.}\label{fig:gaussianSeriesDiffusionSparsity}
	\end{subfigure}
	\caption{Solving diffusion equation with Gaussian series data.}
	\label{fig:gaussianSeries}
\end{figure}
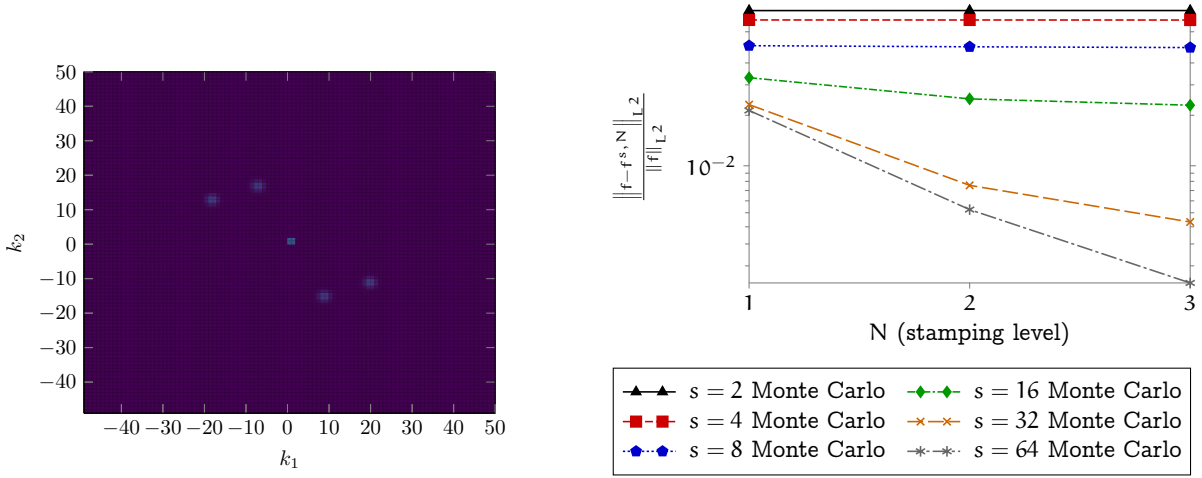

Here, the stamping level does not affect convergence until the sparsity is above $ s \geq 16 $.
This demonstrates the tradeoff between sparsity and stamping level in regards to the error bound \eqref{eq:DrivingErrorFactor}.
Until the SFT is able to capture enough useful information in $ \hat{ a } $, the $ \norm{a - a^s}_{ L^\infty } $ in the error bound dominates.
Eventually, this factor is reduced far enough that the stamping term becomes apparent.

We provide another example, where sparsity is fixed at $ s = 16 $, and dimension and stamping level are increased.
Again we solve \eqref{eq:PDELowSparsity} with the diffusion coefficient replaced by the two-term Gaussian series $ a = c_0 + G^{ \mathcal{I} }_{ \vec{ c }, \vec{ r } } $, where
\begin{equation*}
	\mathcal{I} \sim \Unif \left( \left( [-249, 250]^d \cap \Z^d \right)^2  \right), \quad \vec{ c } \sim \Unif\left([-1, 1]^2\right), \quad \vec{ r } = 1.1^d \vec{ 1 }, \quad c_0 = 10 \left\lceil \norm{ \vec{ c } }_2 \right\rceil,
\end{equation*}
and $ \vec{ c } $ and $ c_0 $ are not redrawn across test cases.
The bandwidth of the SFT is set to $ 1\,000 $ to again account for the potentially widened Fourier transform of $ a $.
With a $ 1\,000 $ point Monte Carlo approximation of the proxy error, the results are given in Figure~\ref{fig:gaussianSeriesDiffusionStamp}.

\begin{figure}[ht]
	\centering
	\begin{tikzpicture}[scale=1]
		\datavisualization [
			scientific axes={width=0.5\linewidth},
			x axis={attribute = n,
				label={$ N $ (stamping level)},
				ticks={few}},
			y axis={attribute = y,
				label={$ \frac{\norm{f - f^{ s, N } }_{ L^2 }}{\norm{ f }_{ L^2 }}$ },
				logarithmic,
				ticks={minor steps between steps=9,
					tick typesetter/.code=\nomantissatypeset{##1}}},
			legend={below, down then right, columns = 2, matrix node style={draw}},
			legend entry options/default label in legend path/.style=
				straight label in legend line,
			visualize as line/.list={
				gridErrors2,
				gridErrors4,
				gridErrors8,
				gridErrors16
			},
			gridErrors2={label in legend={text={$d = 2$ Monte Carlo}}},
			gridErrors4={label in legend={text={$d = 4$ Monte Carlo}}},
			gridErrors8={label in legend={text={$d = 8$ Monte Carlo}}},
			gridErrors16={label in legend={text={$d = 16$ Monte Carlo}}},
			style sheet/.list={vary dashing, strong colors, vary mark},
		]
		data [set=gridErrors2, headline={n, y}, read from file=Results/gaussianSeriesDiffusionDimension/gridErrors_d2_n_y.csv]
		data [set=gridErrors4, headline={n, y}, read from file=Results/gaussianSeriesDiffusionDimension/gridErrors_d4_n_y.csv]
		data [set=gridErrors8, headline={n, y}, read from file=Results/gaussianSeriesDiffusionDimension/gridErrors_d8_n_y.csv]
		data [set=gridErrors16, headline={n, y}, read from file=Results/gaussianSeriesDiffusionDimension/gridErrors_d16_n_y.csv];
	\end{tikzpicture}
\caption{Approximate proxy error solving \eqref{eq:PDELowSparsity} with Gaussian series diffusion coefficient with $ d = 2$, $ 4 $, $ 8 $, $ 16 $ and $ N = 1, \ldots, 5 $.}\label{fig:gaussianSeriesDiffusionStamp}
\end{figure}
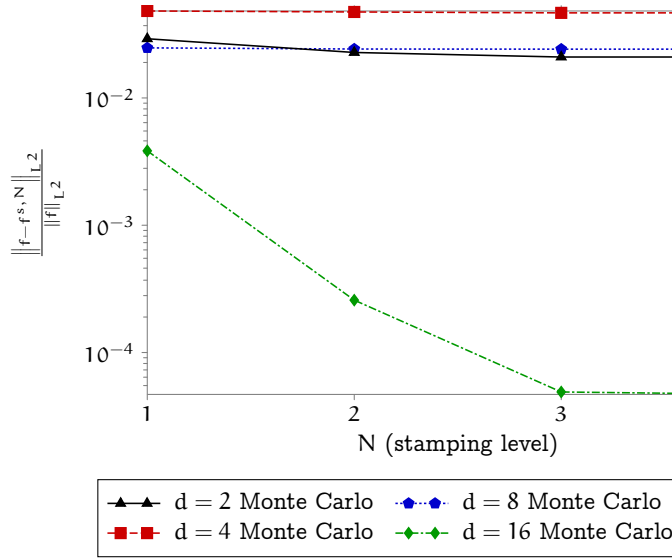

Here we observe much the same behavior as the previous test case.
This is due to the fact that the dimension additionally drives the sparsity of the Gaussian Fourier transforms based on the choice of dilation $ \vec{ r } = 1.1^d \vec{ 1 } $.
In additional experiments performed at higher dimensions (not pictured here), this factor results in numerical instability and the approximation error blows up.
We also see that the $ d = 2 $ and $ d = 4 $ examples are swapped from their assumed positions (and the $ d = 2 $ case even mildly benefits from increased stamping level).
This is attributed to the random draw of the frequency locations affecting the proxy error as well as the SFT algorithm performing better in lower dimensions when all parameters are fixed.

\subsection{Three-dimensional exactly sparse advection-diffusion-reaction equation}\label{sub:three_dimensional_adr}

We now extend our numerical experiments to the situation of a three-dimensional advection-diffusion-reaction equation.
See Remark~\ref{rem:ADR} for the PDE setup and necessary algorithmic modifications.

Numerically, we work with the following exactly sparse data:
\begin{equation}
	\label{eq:ADRTerms}
	\begin{gathered}
		a(\vec{ x }) = \hat{ a }_\vec{ 0 } + \sum_{ \vec{ k } \in \mathcal{I}_a^\mathrm{sine} } c_{a, \vec{ k } }^\mathrm{sine} \sin(2\pi \vec{ k } \cdot \vec{ x }) + \sum_{ \vec{ k } \in \mathcal{I}_a^\mathrm{cosine} } c_{a, \vec{ k } }^\mathrm{cosine} \cos(2\pi \vec{ k } \cdot \vec{ x })\\
		b_j(\vec{ x }) = \sum_{ \vec{ k } \in \mathcal{I}_{b_j}^\mathrm{sine} } c_{b_j, \vec{ k } }^\mathrm{sine} \sin(2\pi \vec{ k } \cdot \vec{ x }) + \sum_{ \vec{ k } \in \mathcal{I}_{ b_j }^\mathrm{cosine} } c_{b_j, \vec{ k } }^\mathrm{cosine} \cos(2\pi \vec{ k } \cdot \vec{ x }) \text{ for all } j = 1,2,3\\
		c(\vec{ x }) = \hat{ c }_\vec{ 0 } + \sum_{ \vec{ k } \in \mathcal{I}_c^\mathrm{sine} } c_{c, \vec{ k } }^\mathrm{sine} \sin(2\pi \vec{ k } \cdot \vec{ x }) + \sum_{ \vec{ k } \in \mathcal{I}_c^\mathrm{cosine} } c_{c, \vec{ k } }^\mathrm{cosine} \cos(2\pi \vec{ k } \cdot \vec{ x })\\
		f(\vec{ x }) = \sum_{ \vec{ k } \in \mathcal{I}_f^\mathrm{sine} } c_{f, \vec{ k } }^\mathrm{sine} \sin(2\pi \vec{ k } \cdot \vec{ x }) + \sum_{ \vec{ k } \in \mathcal{I}_f^\mathrm{cosine} } c_{f, \vec{ k } }^\mathrm{cosine} \cos(2\pi \vec{ k } \cdot \vec{ x }),
	\end{gathered}
\end{equation}
where
\begin{equation*}
	\begin{gathered}
		\abs{\mathcal{I}_a^\mathrm{sine}} = \abs{\mathcal{I}_a^\mathrm{cosine}} = 2 \\
		\abs{\mathcal{I}_{ b_j }^\mathrm{sine}} = \abs{\mathcal{I}_{ b_j }^\mathrm{cosine}} =\abs{\mathcal{I}_{ c }^\mathrm{sine}} = \abs{\mathcal{I}_\vec{ c }^\mathrm{cosine}} = 5  \text{ for all } j = 1,2,3\\
		\abs{\mathcal{I}_f^\mathrm{sine}} = 2, \text{ and } \abs{\mathcal{I}_f^\mathrm{cosine}} = 3.
	\end{gathered}
\end{equation*}
In total, there are $ 45 $ terms composing the differential operator, and $ 5 $ terms composing the forcing function.
Each frequency is randomly drawn from $ \Unif([-49, 50]^3 \cap \Z^3) $ and each coefficient for $ a $ and $ f $ from $ \Unif([-1, 1]) $.
The coefficients for $ \vec{ b } $ and $ c $ are drawn from $ \Unif([0, 1]) $.
To ensure well-posedness, $ \hat{ a }_\vec{ 0 } = 4 \left \lceil \sqrt{ \norm{ c_a^\mathrm{sine} }_2^2 + \norm{ c_a^\mathrm{cosine} }_2^2 } \right \rceil $, and $ \hat{ c }_\vec{ 0 } = 4 \left \lceil \sqrt{ \norm{ c_c^\mathrm{sine} }_2^2 + \norm{ c_c^\mathrm{cosine} }_2^2 } \right \rceil  $.
The bandwidth of the SFT is set to $ K = 100 $ and consider sparsity levels $ s = 2 $ and $ s = 5 $.
Due to the large size of the stamp, we only consider stamping levels $ N = 1,2$.

\begin{table}[ht]
	\centering
	\begin{filecontents*}{./ADR3DError.csv}
s,n,err_fourierError,err_gridError
2,1,0.517860465143161,0.517860465143161
2,2,0.517541002219515,0.517541002219515
5,1,0.0543339433509572,0.0543339433509572
5,2,0.0313552426435283,0.0313552426435283
\end{filecontents*}
\pgfplotstabletypeset[col sep=comma,
every head row/.style={before row=\toprule & & \multicolumn{2}{c}{$ \norm{ f - f^{ s, N } }_{ L^2}/ \norm{ f }_{ L^2 } $}\\\cmidrule{3-4}},
		every last row/.style={after row=\bottomrule},
		every first column/.style={column type/.add={}{}},
		every column/.style={column type/.add={}{}},
		columns/s/.style={
			column name = $ s $,
			assign cell content/.code={\ifodd\pgfplotstablerow
					\pgfkeyssetvalue{/pgfplots/table/@cell content}{
						\cmidrule{2-4}}
				\else
					\pgfkeyssetvalue{/pgfplots/table/@cell content}{
						\midrule \multirow{2}{*}{##1}}
				\fi
			}
		},
		columns/n/.style={
			column name = $ N $,
		},
		columns/err_fourierError/.style={
			column name = exact,
			fixed, fixed zerofill, precision=3
		},
		columns/err_gridError/.style={
			column name = Monte Carlo,
			fixed, fixed zerofill, precision=3,
		},
	]{./ADR3DError.csv}
\caption{Error in approximating solution to ADR equation \eqref{eq:ADR}.}
	\label{tab:ADRError}
\end{table}

\begin{figure}[ht]
	\begin{subfigure}[t]{0.3\linewidth}
		\begin{tikzpicture}[scale=0.6]
			\begin{axis}[
				view={0}{90},
				xlabel=$x_2$,
				ylabel=$x_3$
			]
				\addplot3 [
					surf, 
					mesh/rows=64,
					colormap name={viridis},
				] table [col sep=comma] {Figures/fRecoveredVals_N128_n1_s2_x0.492188_2d.csv};
			\end{axis}	
		\end{tikzpicture}
		\caption{Slice through $ f^{ 2, 1 } $.}\label{fig:ADR3DRightHandSideApproximateN1Slice}
	\end{subfigure}
	\begin{subfigure}[t]{0.3\linewidth}
		\begin{tikzpicture}[scale=0.6]
			\begin{axis}[
				view={0}{90},
				xlabel=$x_2$,
				ylabel=$x_3$
			]
				\addplot3 [
					surf, 
					mesh/rows=64,
					colormap name={viridis},
				] table [col sep=comma] {Figures/fRecoveredVals_N128_n2_s10_x0.492188_2d.csv};
			\end{axis}	
		\end{tikzpicture}
		\caption{Slice through $ f^{ 10, 2 } $.}\label{fig:ADR3DRightHandSideApproximateN2Slice}
	\end{subfigure}
	\begin{subfigure}[t]{0.3\linewidth}
		\begin{tikzpicture}[scale=0.6]
			\begin{axis}[
				view={0}{90},
				xlabel=$x_2$,
				ylabel=$x_3$
			]
				\addplot3 [
					surf, 
					mesh/rows=64,
					colormap name={viridis},
				] table [col sep=comma] {Figures/fVals_N128_x0.492188_2d.csv};
			\end{axis}	
		\end{tikzpicture}
		\caption{Slice through $ f $.}\label{fig:ADR3DRightHandSideSlice}
	\end{subfigure}
	\caption{Samples of $ f^{ 10, 2 } $ and $ f $ on the $ x_1 = 63 / 128 $ plane.}
	\label{fig:ADR3DRightHandSide}
\end{figure}
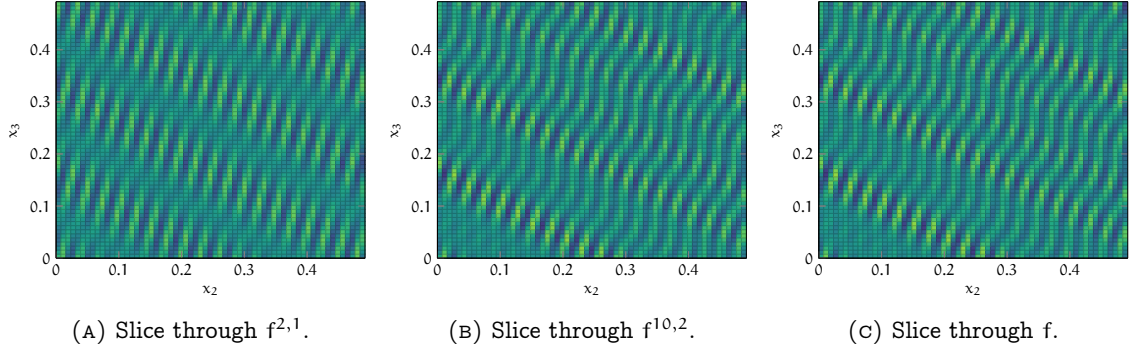

The resulting true and Monte Carlo proxy error (sampled over $ 1\,000 $ points) is given in Table~\ref{tab:ADRError}.
Additionally, Figure~\ref{fig:ADR3DRightHandSide} shows a portion of a slice through $ f $ as well as $ f^{ 2, 1 } $ and $ f^{ 10, 2 } $ which are computed by passing $ u^{ 2, 1 } $ and $ u^{ 10, 2 } $ through the differential operator.

We note that $ f^{ 10, 2 } $ and $ f $ appear qualitatively indistinguishable.
However, since the sparsity level, $ s = 2 $, used to compute $ u^{ 2, 1 } $ is lower than the sparsity of any term in \eqref{eq:ADRTerms}, $ f^{ 2, 1 } $ loses some of characteristics of the original source term.
Though it captures some of the true behavior in both larger scales (e.g., the oscillations moving in the northeast direction) and finer scales (e.g., the oscillations moving in the southeast direction), some interfering modes which produce the ``wavy'' effect are left out.
This is supported by the relative errors reported in Table~\ref{tab:ADRError}.
Note also that the stamping level affects the convergence in $ s = 5 $ case, but not the $ s = 2 $ case.
This is due to the sparsity related errors in \eqref{eq:DrivingErrorFactor} overwhelming the stamping term until the SFT approximations of the data are accurate enough.

\appendix

\section{Stamp set cardinality bound}\label{sec:stamp_set_cardinality_bound}

We begin by proving the following combinatorial upper bound for the cardinality of a stamp set.

\begin{lem}
	\label{lem:StampSetCardinality}
	Suppose that $ \vec{ 0 } \in \supp(\hat{ a }) $, $ \supp(\hat{ a }) = -\supp(\hat{ a }) $, $ \abs{ \supp(\hat{ a }) } = s $.
	Then
	\begin{equation}
		\label{eq:CombinatorialStampSetCardinality}
		\abs{ \mathcal{S}^N[\hat{ a }](\supp(\hat{ f })) } \leq \abs{ \supp(\hat{ f }) } \sum_{ n = 0 }^N \sum_{ t = 0 }^{ \min(n, (s-1)/2) } 2^{ t } \binom{(s-1)/2}{t} \binom{n - 1}{t - 1}.
	\end{equation}
\end{lem}
\begin{proof}
	We begin by separating $ \mathcal{S}^N $ into the disjoint pieces
	\begin{equation*}
		\mathcal{S}^N = \bigsqcup_{ n = 0 }^{ N } \left(\mathcal{S}^n \setminus \left( \bigcup_{ i = 0 }^{ n - 1 } \mathcal{S}^i \right) \right)
	\end{equation*}
	and computing the cardinality of each of these sets (where we take $ S^{ -1 } = \emptyset $).
	If $ \vec{ k } \in \mathcal{S}^n \setminus \left( \cup_{ i = 0 }^{ n - 1 } \mathcal{S}^i \right) $, then we are able to write $ \vec{ k } $ as
	\begin{equation}
		\label{eq:StampFrequencySum}
		\vec{ k } = \vec{ k }_f + \sum_{ m = 1 }^n \vec{ k }_a^m
	\end{equation}
	where $ \vec{ k }_f \in \supp(\hat{ f }) $ and $ \vec{ k }_a^m \in \supp(\hat{ a }) \setminus \{ 0 \} $ for all $ m = 1, \ldots, n $.
	Additionally, since $ \vec{ k } $ is not in any earlier stamping sets, this is the smallest $ n $ for which this is possible.
	In particular, it is not possible for any two frequencies in the sum to be negatives of each other resulting in pairs of cancelled terms.

	With this summation in mind, arbitrarily split $ \supp( \hat{ a } ) \setminus \{ \vec{ 0 } \} $ into $ A \sqcup -A $ (i.e., place all frequencies which do not negate each other into $ A $ and their negatives in $ -A $).
	By collecting like frequencies that occur as a $ \vec{ k }_a^m $ term in \eqref{eq:StampFrequencySum}, we can rewrite this sum as
	\begin{equation}
		\label{eq:StampFrequencyEnumeration}
		\vec{ k } = \vec{ k }_f + \sum_{ \vec{ k }_a \in A } s(\vec{ k }, \vec{ k }_a) m(\vec{ k }, \vec{ k }_a) \vec{ k }_a,
	\end{equation}
	where the sign function $ s(\vec{ k }, \vec{ k }_a) $ is given by
	\begin{equation*}
		s(\vec{ k }, \vec{ k }_a) := 
		\begin{cases}
			1 &\text{if $ \vec{ k }_a $ is a term in the summation \eqref{eq:StampFrequencySum}}\\
			-1 &\text{if $ -\vec{ k }_a $ is a term in the summation \eqref{eq:StampFrequencySum}}\\
			0 &\text{otherwise}
		\end{cases}
	\end{equation*}
	and the multiplicity function $ m(\vec{ k }, \vec{ k }_a) $ is defined as the number of times that $ \vec{ k }_a $ or $ -\vec{ k }_a $ appears as a $ \vec{ k }_a^m $ term in \eqref{eq:StampFrequencySum}.
	Letting $ \vec{ s }(\vec{ k }) := (s(\vec{ k }, \vec{ k }_a))_{ k_a \in A } $ and $ \vec{ m }(\vec{ k }) := (m(\vec{ k }, \vec{ k }_a))_{ k_a \in A } $, we can then identify any $ \vec{ k } \in \mathcal{S}^n \setminus \left( \cup_{ i = 0 }^{ n - 1 } \mathcal{S}^i \right) $ with the tuple 
	\begin{equation*}
		(\vec{ k }_f, \vec{ s }(\vec{ k }), \vec{ m }(\vec{ k })) \in \supp(\vec{ f }) \times \{-1, 0, 1\}^{ A } \times \{0, \ldots, n\}^A.
	\end{equation*}
	Upper bounding the number of these tuples that can correspond to a value of $ \vec{ k } \in \mathcal{S}^n \setminus \left( \cup_{ i = 0 }^{ n - 1 } \mathcal{S}^i \right) $ will then upper bound the cardinality of this set.

	Since any $ \vec{ k }_f \in \supp(\hat{ f }) $ can result in a valid $ \vec{ k } $ value, we will focus on the pairs of sign and multiplicity vectors.
	Define by $ T_n \subset \{-1, 0, 1\}^A \times \{0, \ldots, n\}^A $ the set of valid sign and multiplicity pairs that can correspond to a $ \vec{ k } \in \mathcal{S}^n \setminus \left( \cup_{ i = 0 }^{ n - 1 } \mathcal{S}^i \right)  $.
	In particular, for $ (\vec{ s }, \vec{ m }) \in T_n $, $ \norm{ \vec{ m } }_1 = n $ and $ \supp(\vec{ s } ) = \supp( \vec{ m } ) $.
	Thus, we can write
	\begin{equation*}
		T_n \subset \bigsqcup_{ t = 0 }^{ \min(n, |A|) } \left\{ (\vec{ s }, \vec{ m }) \in \{-1, 0, 1\}^A \times \{0, \ldots, n\}^A \mid \norm{ \vec{ m } }_1 = n \text{ and } |\supp(\vec{ s })| = |\supp(\vec{ m })| = t \right\}.
	\end{equation*}
	This inner set then corresponds to the $ t $-partitions of the integer $ n $ spread over the $ |A| $ entries of $ \vec{ m } $ where each non-zero term is assigned a sign $ -1 $ or $ 1 $.
	The cardinality is therefore $ 2^t \binom{ |A| }{ t }\binom{n - 1}{t - 1} $: the first factor is from the possible sign options, the second is the number of ways to choose the entries of $ \vec{ m } $ which are nonzero, and the last is the number of $ t $-partitions of $ n $ which will fill the nonzero entries of $ \vec{ m } $.
	Noting that $ |A| = \frac{s - 1}{2} $, our final cardinality estimate is
	\begin{align*}
		\abs{\mathcal{S}^N} 
			&= \sum_{ n = 0 }^{ N } \abs{\mathcal{S}^n \setminus \left( \bigcup_{ i = 0 }^{ n - 1 } \mathcal{S}^i \right)}\\
			& \leq \sum_{ n = 0 }^N \abs{\supp(\hat{ f })} |T_n|\\
			& \leq \abs{ \supp(\hat{ f }) } \sum_{ n = 0 }^N \sum_{ t = 0 }^{ \min(n, (s-1)/2) } 2^{ t } \binom{(s-1)/2}{t} \binom{n - 1}{t - 1}
	\end{align*}
	as desired.
\end{proof}

Though this upper bound is much tighter than the one given in the main text, it is harder to parse.
As such, we simplify it to the bound presented in Lemma~\ref{lem:StampSizeUpperBound}, restated here for convenience.
{
\renewcommand{\thelem}{\ref{lem:StampSizeUpperBound}}
\begin{lem}
Suppose that $ \vec{ 0 } \in \supp(\hat{ a }) $, $ \supp(\hat{ a }) = -\supp(\hat{ a }) $, and $ \abs{ \supp(\hat{ f }) } \leq \abs{ \supp(\hat{ a }) } = s $
	Then
	\begin{equation*}
		\abs{ \mathcal{S}^N[\hat{ a }](\supp (\hat{ f }))} \leq 7 \max(s, 2N + 1)^{ \min(s, 2N + 1)  }.
	\end{equation*}
\end{lem}
\addtocounter{lem}{-1}
}
\begin{proof}
	Let $ r = (s - 1) / 2 $.
	We consider two cases:
	\begin{description}
		\item[Case 1: $ r \geq N $] 
			We estimate the innermost sum of \eqref{eq:CombinatorialStampSetCardinality}.
			Since $ r \geq N \geq n$, $ \min(n, (s - 1) / 2) = n $.
			By upper bounding the binomial coefficients with powers of $ r $, we obtain
			\begin{align*}
				\sum_{ t = 0 }^{ n } 2^t \binom{r}{t} \binom{n-1}{t-1}
					&\leq \sum_{ t = 0 }^n 2^t (r^t)^2 \\
					&\leq 2(2r^2)^{ n }
			\end{align*}
			where the second estimate follows from the approximating the geometric sum.
			Again, bounding the next geometric sum by double the largest term, we have
			\begin{equation*}
				\abs{ \mathcal{S}^N } \leq \abs{ \supp(\hat{ f }) } \sum_{ n = 0 }^N 2(2s^2)^n \leq (2r + 1) 4(2r^2)^N \leq 2(2r + 1)^{ 2N + 1 } = s^{ 2N + 1 }.
			\end{equation*}
		\item[Case 2: $ r < N $] 
			Bounding the innermost sum of \eqref{eq:CombinatorialStampSetCardinality} proceeds much the same way as Case~1, but we must first split the outermost sum into the first $ r + 1 $ terms and last $ N - r $ terms.
			Working with the first terms, we find
			\begin{equation*}
				\sum_{ n = 0 }^r \sum_{ t = 0 }^{ n } 2^t \binom{r}{t} \binom{n-1}{t-1} \leq 4(2r^2)^r
			\end{equation*}
			using the argument in Case~1.
			Now, we bound
			\begin{align*}
				\sum_{ n = r + 1 }^N \sum_{ t = 0 }^{ r } 2^t \binom{r}{t} \binom{n-1}{t-1} 
					&\leq \sum_{ n = r + 1 }^N 2(2(n - 1)^2)^r \\
					&\leq 2^{ r + 1 } \int_{ r }^{ N } n^{ 2r } \, dn\\
					&\leq \sqrt{ 2 }\frac{(\sqrt{ 2 }N)^{ 2r + 1 }}{2r + 1}.
			\end{align*}
			Thus,
			\begin{equation*}
				\abs{ \mathcal{S}^N } \leq \abs{ \supp(\hat{ f }) }\left[ 4(2r^2)^r + \sqrt{ 2 } \frac{(\sqrt{ 2 }N)^{ 2r + 1 }}{2r + 1} \right]  \leq 5\sqrt{ 2 }\left( \sqrt{ 2 }N \right)^s \leq 7 (2N + 1)^s.
			\end{equation*}
		Combining the two cases gives the desired upper bound.
	\end{description}
\end{proof}

\section{Proof of SFT recovery guarantees}\label{sec:proof_of_sft_recovery_guarantees}

We restate the theorem for convenience.

{
\renewcommand{\thethm}{\ref{thm:SFTRecovery}}
\begin{thm}[\cite{gross_sparse_2021}, Corollary 2]
Let $ \mathcal{I} \subset \Z^d$ be a frequency set of interest with expansion defined as $ K := \max_{ j \in \{1, \ldots, d\} } (\max_{ \vec{ k } \in \mathcal{I} } k_j - \min_{ \vec{ l } \in \mathcal{I}} l_j ) $ (i.e., the sidelength of the smallest hypercube containing $ \mathcal{I} $), and $ \Lambda( \vec{ z }, M) $ be a reconstructing rank-1 lattice for $ \mathcal{I} $.

	There exists a fast, randomized SFT which, given $ \Lambda(\vec{ z }, M) $, sampling access to $ g \in L^2 $, and a failure probability $ \sigma \in (0, 1] $, will produce a $ 2s $-sparse approximation $ \hat{ \vec{ g } }^s $ of $ \hat g $ and function $ g^s := \sum_{ \vec{ k } \in \supp( \hat{ \vec{ g }}^s) } \hat g_{ \vec{ k } }^s e_\vec{ k } $ approximating $ g $ satisfying
	\begin{align*}
		\norm{ g - g^s }_{ L^2 } \leq \norm{ \hat g - \hat{ \vec{ g } }^s }_{ \ell^2 } 
			&\leq (25 + 3K) \left[ \frac{\norm{\hat{ g }\restrict{ \mathcal{I} } - (\hat{ g }\restrict{ \mathcal{I} })_s^\mathrm{opt}}_1}{\sqrt{ s }} + \sqrt{ s } \norm{ \hat{ g } - \hat{ g }\restrict{ \mathcal{I} } }_1 \right]
	\end{align*}
	with probability exceeding $ 1 - \sigma $.
	If $ g \in L^\infty $, then we additionally have
	\begin{equation*}
		\norm{ g - g^s }_{ L^\infty } \leq \norm{ \hat g - \hat{ \vec{ g } }^s }_{ \ell^1 } \leq (33 + 4K) \left[ \norm{ \hat{ g } \restrict{ \mathcal{I} } - (\hat{ g }\restrict{ \mathcal{I} })_s^\mathrm{opt} }_1 + \norm{ \hat{ g } - \hat{ g }\restrict{ \mathcal{I} } }_1 \right]
	\end{equation*}
	with the same probability estimate.
	The total number of samples of $ g $ and computational complexity of the algorithm can be bounded above by
	\begin{equation*}
		\mathcal{O} \left( d s \log^3( d K M) \log \left( \frac{d K M }{\sigma} \right)  \right).
	\end{equation*}
\end{thm}
\addtocounter{thm}{-1}
}
\begin{proof}
	The $ L^2 $ upper bound is mostly the same as the original result.
	We are not considering noisy measurements here which removes the $ \sqrt{ s } e_\infty $ term from that result (though, this could be added back in if desired).
	Additionally, we have upper bounded $ \norm{ \hat{ g } - \hat{ g }\restrict{ \mathcal{I} } }_2 $ by $ \sqrt{s} \norm{ \hat{ g } - \hat{ g }\restrict{ \mathcal{I} } }_1 $ adding one to the constant.

	The $ L^\infty $ / $ \ell^1 $ bound was not given in the original paper, but can be proven using the same techniques.
	In particular, replacing the $ \ell^2 $ norm by the $ \ell^1 $ norm in \cite[Lemma 4]{gross_sparse_2021} has the effect of replacing all $ \ell^2 $ norms with $ \ell^1 $ norms and replacing $ \sqrt{ 2s } $ by $ 2s $.
	This small change cascades through the proof of Property~3 in \cite[Theorem 2]{gross_sparse_2021} (again, with $ \ell^2 $ norms replaced by $ \ell^1 $ norms) to produce the univariate $ \ell^1 $ upper bound (in the language of the original paper)
	\begin{equation*}
		\norm{ \hat{ \vec{ a } } - \vec{ v } }_1 \leq \norm{ \hat{ \vec{ a } } - \hat{ \vec{ a } }_{ 2s }^\mathrm{opt} }_1 + (16 + 6 \sqrt{ 2 }) \left(\norm{ \hat{ \vec{ a } } - \hat{ \vec{ a } } _{ s }^\mathrm{opt} }_1 + s (\norm{ \hat{ a } - \hat{ \vec{ a } } }_1 + \norm{ \mu }_\infty)\right) =: \eta_1.
	\end{equation*}
	
	A similar logic applies to revising the proof of \cite[Lemma 1]{gross_sparse_2021}.
	Equation~(4) with all $ \ell^2 $ norms replaced by $ \ell^1 $ norms is derived the same way, and the first term is upper bounded by the maximal entry of the vector multiplied by the number of elements without the square root.
	The remainder of the proof carries through without change which leads to a final error estimate of
	\begin{equation*}
	\norm{ \vec{ b } - c }_{ \ell^2 } \leq (\beta + \eta_\infty) \max(s - \abs{ \mathcal{S}_\beta}, 0 ) +  \eta_1 + \norm{ c\restrict{ \mathcal{I} } - c\restrict{ \mathcal{S}_\beta } }_{ 1 } + \norm{ c - c\restrict{ \mathcal{I} } }_1.
	\end{equation*}
	Finally, the proof of \cite[Corollary 2]{gross_sparse_2021} follows using the same logic as the original substituting these revised upper bounds.
\end{proof}

\section*{Acknowledgements}
This work was supported in part by the National Science Foundation Award Numbers DMS 2106472 and 1912706.
This work was also supported in part through computational resources and services provided by the Institute for Cyber-Enabled Research at Michigan State University.
We thank Lutz K\"ammerer for helpful discussions related to random rank-1 lattice construction and Ben Adcock and Simone Brugiapaglia for motivating discussions related to compressive sensing and high-dimensional PDEs.

\bibliographystyle{amsplain}
\bibliography{References}
	
\end{document}